\documentclass[12pt,reqno]{amsart}
\usepackage{amsmath}
\usepackage{amscd}
\usepackage{graphics}
\usepackage{latexsym}
\usepackage{color}
\usepackage{verbatim}
\usepackage{extarrows}
\usepackage{latexsym}
\usepackage{enumerate}
\usepackage{arydshln}
\usepackage{soul}
\usepackage{tikz,pgfplots}
\usepackage{tikz-cd}
\usepackage{ulem}
\usepackage{MnSymbol}

\usepackage{marvosym}
\usepackage{mathrsfs}

\usepackage{hyperref}
\hypersetup{
    colorlinks=true,
    linkcolor=red, 
    filecolor=red, 
    urlcolor=red,
   citecolor=red }
\usepackage[alphabetic]{amsrefs}

\definecolor{rose}{rgb}{0.9, 0.17, 0.31}
\definecolor{darkgreen}{HTML}{3CB50F}

\usetikzlibrary{arrows,calc,backgrounds}
\usepackage[margin=1.3in]{geometry}
\usepackage[all]{xy}
\theoremstyle{plain}
\newtheorem{theorem}{Theorem}[section]

\newtheorem{cor}[theorem]{Corollary}
\newtheorem{lem}[theorem]{Lemma}
\newtheorem{prop}[theorem]{Proposition}

\theoremstyle{definition}
\newtheorem{definition}[theorem]{Definition}
\newtheorem{notat}[theorem]{Notation}

\newtheorem{rmk}[theorem]{Remark}

\theoremstyle{theorem}

\newtheorem{mocktheorem}{Theorem}
\newtheorem{mockcorollary}[mocktheorem]{Corollary}

\theoremstyle{remark}

\newcommand{\ZZ}{\mathbb{Z}}

\newcommand{\PP}{\mathbb{P}}
\newcommand{\EFF}{\mathrm{Eff}}

\newcommand{\NEF}{\mathrm{Nef}}
\newcommand{\MOV}{\mathrm{Mov}}

\newcommand{\PIC}{\mathrm{Pic}}

\newcommand{\PPP}{\PP^{3}}

\newcommand{\OO}{\mathcal O}
\newcommand{\OOT}{\mathcal{O}_{\mathbb{P}^2}}

\usepackage{graphicx}
\usepackage{subcaption}
\usepackage{wrapfig}

\def\QQ{{\mathbb{Q}}}

\def\PP{{\mathbb P}}

\title{Explicit birational geometry of determinantal quartic 3-folds
}

\author{Manuel Leal and C\'esar Lozano Huerta and Montserrat Vite}

\begin{document}

\address{Universidad Nacional Aut\'onoma de M\'exico\\
Instituto de Matem\'aticas, Unidad Oaxaca \\
Oaxaca de Ju\'arez, M\'exico.}
\email{manuelleal@im.unam.mx}
\email{lozano@im.unam.mx} 

\address{Centro de Ciencias Matem\'aticas, Universidad Nacional Aut\'onoma de M\'exico  \\ Morelia, Michoac\'an, M\'exico.}
\email{montserrat@matmor.unam.mx}

\maketitle

\begin{abstract} 
    A general linear determinantal quartic in $\mathbb{P}^4$ is nodal, non-$\mathbb{Q}$-factorial and rational. We show that the family $\mathcal{F}$ of such quartics also contains rational $\mathbb{Q}$-factorial quartics, and that a generic member of $\mathcal{F}$ can specialize to a rational non-$\mathbb{Q}$-factorial double quadric.
    We prove that the birational geometry of these three types of 3-folds is governed by the extrinsic geometry of a curve $C\subset \mathbb{P}^3$ of degree 10 and genus 11.
\end{abstract}

\section*{Introduction}

\noindent
 Cubic hypersurfaces in $\PP^4$ occupy a distinguished place in algebraic geometry, and their intermediate jacobians are a among their most studied features. For a smooth cubic treefold $X$, the intermediate Jacobian is the Prym variety of a smooth plane curve C endowed with an odd theta characteristic. This description not only allows one to study $X$ but also its moduli, via the geometry of $C$ \cite{radu}. 
 
 \medskip\noindent
 We take a similar approach in degree four. However, the intermediate Jacobian of a smooth quartic threefold is intricate, so we restrict to the following class of quartics.

\noindent

\medskip\noindent
A determinantal quartic in $\PP^4$ is the zero locus of the determinant of a matrix of polynomial forms. Under mild constraints on these forms, there exist five irreducible families of such quartics (Lemma \ref{5families}). 
We focus on one of them: the family $\mathcal{F}$ of \textit{linear} determinantal quartics. A general member of $\mathcal{F}$ has 20 nodes and is not $\QQ$-factorial. This means that such nodes are in special position, but remarkably, $\mathcal{F}$ is not contained in any larger family of 20-nodal quartics (Remark \ref{max1}).

\medskip\noindent
Let us examine a general $Y \in \mathcal{F}$. This quartic is rational, and its determinantal presentation yields a small resolution $f:X\to Y$, where $X$ is isomorphic to the blow-up of $\PP^3$ along a \textit{general} curve $C$ of degree 10 and genus 11 (Section \ref{construction}). We may thus study $Y$, and aim to gain understanding of its moduli, via the curve $C$.

\medskip\noindent 
Our main result, Theorem \ref{mocktheorem::introThm}, carries out this strategy by running the Minimal Model Program on $X$, provided that $C$ remains smooth and satisfies the conditions below.

\medskip\noindent
The Hilbert scheme of curves of degree 10 and genus 11 has a unique irreducible component $\mathcal{H}_{10,11}$, which generically parametrizes smooth and irreducible curves. We show that the smooth curves in $\mathcal{H}_{10,11}$ must be of three types (Theorem \ref{prop::smoothCurves}): arithmetically Cohen-Macaulay (ACM) curves (which form a dense open subset of $\mathcal{H}_{10,11}$), 
\textit{semicanonical} curves or curves contained in a cubic surface. 

\medskip\noindent
Theorem \ref{mocktheorem::introThm} summarizes our study of $X=Bl_C\PP^3$, where $C$ is smooth and thus belongs to one of the three types above. Its proof is in Theorems \ref{teorema1}, \ref{thm3} and \ref{teorema2}. Even though each case has distinctive features, there is a uniform statement.

\begin{mocktheorem}\label{mocktheorem::introThm}
Let $C\in \mathcal{H}_{10,11}$ be general in one of the three families described above parametrizing smooth and irreducible curves. The effective cone of $ X=Bl_C\PP^3$ has exactly four stable base locus chambers and the birational models induced by each of them can be described explicitly in terms of the extrinsic geometry of $C$.
\end{mocktheorem}

\medskip\noindent
When $C\in \mathcal{H}_{10,11}$ is general $i.e.$, ACM, the anti-canonical model of $X=Bl_C\PP^3$ is a linear determinantal quartic $Y$. This means that $X$ and $Y$ belong to families which can be studied via $C$. Since $C$ has exactly two specializations that are smooth but not ACM (Theorem \ref{prop::smoothCurves}), then through the anti-canonical model of $Bl_C\PP^3$, there are two distinguished specializations of $Y$:

\begin{mockcorollary}\label{mockcorollary::introCor}
The general $Y\in \mathcal{F}$ can specialize to a non-$\mathbb{Q}$-factorial double quadric $Y_0\to Q\subset \PP^4$ ramified along a degree 8 surface with twenty nodes; or to a $\mathbb{Q}$-factorial quartic with ten nodes and one elliptic singularity.
\end{mockcorollary}

\medskip\noindent
The intermediate Jacobian $IJ(X)$ of $X$ coincides with the Jacobian of $C$ as a principally polarized abelian variety \cite[\S 3]{GC}. Consequently, the specializations above correspond to variations of $IJ(X)$, reflecting changes in moduli rather than merely extrinsic specializations. For instance, a double quadric $Y_0\to Q\subset \PP^4$ from Corollary \ref{mockcorollary::introCor} no longer belongs to $\mathcal{F}$, analogous to what occurs in $\mathcal{M}_3$ when a smooth plane curve of degree 4 specializes to a double conic
(further discussion: Section \ref{S5}). 

\medskip\noindent
 As the generic $Y\in \mathcal{F}$ specializes to a $\mathbb{Q}$-factorial quartic a dramatic degeneration occurs: ten nodes coalesce into an elliptic singularity. This geometry, though somewhat intricate, can be described explicitly thanks to the smoothness of $C$, as explained below.
 
\medskip\noindent
Theorem \ref{mocktheorem::introThm} summarizes a wealth of geometric information which we consider is the central contribution of the paper. Its proof is a detailed account of the birational geometry of $X$ in terms of the extrinsic geometry of $C\subset \PP^3$. For example, a \textit{general} $C\in \mathcal{H}_{10,11}$ has exactly twenty 4-secant lines whose strict transforms are contracted under the anti-canonical map $f:X\to Y$, producing the twenty double points of $Y$. Theorem \ref{teorema1} describes how $X$ sits in a Sarkisov link that flops these 4-secant lines, encompassing all of its birational models:
\begin{equation}\label{SarkisovIntro}
\xymatrix @!=1pc{
& X\ar[dl] \ar@{-->}[rr]^{flop} \ar[dr]^{f}&  &  X^+ \ar[dl]\ar[dr]&\\
\PP^{3} \ar@/_0.6cm/@{-->}[rrrr]_{\phi} && Y &&  \PP^3} 
\end{equation}

\medskip\noindent
Lemma \ref{lemma11he} shows that the Cremona transformation $\phi$ above is defined by surfaces of degree 11 vanishing with multiplicity 3 along $C$. The non-trivial edge of the effective cone $\EFF(X)$ is spanned by a surface of degree 40, which vanishes to order 11 along $C$ and whose normalization is a ruled surface over $C$.

\medskip\noindent
    As $C$ specializes to a smooth curve $C'$ with $\omega_{C'}=\mathcal{O}_{C'}(2)$, the birational geometry of $X$ can also be explained as above. A distinctive feature of this case is that the anti-canonical linear system realizes $X$ as a double quadric ramified along a surface with twenty nodes that fail to impose independent conditions on cubics (Proposition \ref{prop:: doublequadric}). 
    Theorem \ref{thm3} describes this explicitly.

\medskip\noindent
If $C$ specializes to a smooth curve that lies in a cubic surface $S$, its 4-secant lines also specialize and two $5$-secant lines $l_1,l_2$ show up. In this case, $X$ is the target of a flip, and diagram \eqref{SarkisovIntro} becomes:
\begin{equation*}
\xymatrix @!=1pc{
& X\ar[dl] \ar@{<--}[rr]_{\varphi}^{\mbox{\tiny{flip}}} \ar[dr]^{}&  &  X^{+} \ar[dl]_{}\ar[dr]^{\pi^+}&\\
\PP^{3} && Y && \ \ \  Y_0\subset \ \PP^{4}} 
\end{equation*}
where $X^+$ is a singular terminal threefold which admits a divisorial contraction onto a quartic $Y_0\subset \PP^4$. Indeed, the union $l_1\cup l_2\subset X$ is the flipping locus of $\varphi$, and the cubic surface $S\subset X$ maps to a quintic del Pezzo in $X^+$, which is then contracted to an elliptic singularity of $Y_0$. The anticanonical map is a rational divisorial contraction $\pi^+\circ \varphi^{-1}:X \dashedrightarrow Y_0$. This is the content of Theorem \ref{teorema2}.

\medskip\noindent
Summarizing, these results provide a novel example of how the extrinsic geometry of the curve $C$ offers insight into the birational geometry of a determinantal quartic and its moduli $\mathcal{M}$. In fact, this paper suggests there are two irreducible divisors in $\mathcal{M}$ associated to the only smooth non-ACM curves in $\mathcal{H}_{10,11}$ (Corollary \ref{mockcorollary::introCor}). Understanding $\mathcal{M}$ is our ultimate goal (see Remark \ref{moduliX}), but we postpone further analysis to a later paper in order to keep this one to a reasonable length.

\subsection*{Related results}
\noindent
If $C\in \mathcal{H}_{10,11}$ is general, then the Sarkisov link of diagram \eqref{SarkisovIntro} was previously known \cite{lamy,oslo}. Our contribution is to provide a detailed description of this link in terms of the extrinsic geometry of $C$ using the MMP. Describing the cases when $C$ specializes to a smooth non-ACM curve is new, as far as we are aware. Our main tool in doing so is the Hilbert scheme of curves and the Hartshorne-Rao module.

\subsection*{Organization}
\noindent
Section \ref{sec::2} contains preliminaries. Section \ref{Sec2} studies the Hilbert scheme of curves of degree 10 and genus 11 and identifies its main component, as well as two divisors in it that yield Corollary \ref{mockcorollary::introCor}. This section also shows that the smooth specializations of $C$ discussed above are the only possible ones.

\noindent
Section \ref{sec::3} describes the birational models of the blow-up $X = Bl_C \PP^3$, where $C$ is a general curve in the main component of its Hilbert scheme. Section \ref{S5} addresses the case when $C$ is a smooth semicanonical curve, while Section \ref{S4} treats the case when $C$ lies on a smooth cubic surface.

\section*{acknowledgments}
\noindent We thank Carolina Araujo, Michela Artebani, Alex Massarenti, Pedro Montero and Daniela Paiva for useful conversations. Special thanks to Carolina Araujo and Daniela Paiva for sharing their work with us \cite{caro}, and for pointing out important literature on the subject (see Remarks \ref{AutK3}, \ref{AutK3II}). During the preparation of this note ML was funded by CONAHCYT under the program ``Becas de Posgrado". MV was funded by CONAHCYT, grant ``Estancias Posdoctorales por M\'exico''.

\medskip
\section{ Preliminaries}\label{sec::2}

\noindent
This paper works over an algebraically closed field $\mathbb{K}$ of characteristic 0. Let $C\subset \PP^3$ be a smooth curve of degree $d$ and genus $g$, and let $X:=Bl_C \PP^3$ denote the blow-up of $\PP^3$ along $C$. In this section, we collect elementary features of $X$. 

\medskip\noindent
Since $H^1(X,\OO_X)=0$, linear and numerical equivalence coincide in $X$, so that $\PIC(X)\cong N^1(X)\cong \mathbb{Z}^2$. Consequently, $N^1(X)=\langle H, E\rangle$ is generated by the class $H$, the pullback of a hyperplane section in $\PP^3$; and the exceptional divisor $E$. We use this notation for the rest of the paper.

\medskip\noindent
The effective divisors with class $n H-k E$ correspond bijectively to hypersurfaces in $\PP^3$ of degree $n$ vanishing along $C$ with order at least $k$. If $C'$ is the strict transform of an irreducible curve of degree $n'$ intersecting $C$ on $k'$ points transversally, then $H\cdot C'=n'$ and $E\cdot C'=k'$.

\begin{lem}\cite[Prop. 13.13]{3264} Let $H$ and $E$ be as above. Then we have the following identities $$H^3=1, \quad
        H^2\cdot E=0,\quad
        H\cdot E^2 = -d, \quad
        E^3 = 2-2g-4d.$$

\end{lem}

\medskip\noindent
The Hirzebruch-Riemann-Roch formula yields the following lemma, which will be used in Section \ref{sec::3}.
\begin{lem}\label{lem::chi}
    Suppose that $(d,g)=(10,11)$. Then for all $k,n\in\mathbb{Z}$, the following holds:
    $$\chi(X,nH-kE)=\binom{n+3}{3}-5k(nk+n-2k^2-k+1).$$
\end{lem}

\medskip\noindent
This paper examines how the birational geometry of $X$ depends on $C$. Naturally, the abstract deformations of $X$ are determined by the extrinsic geometry of $C$, as recorded in the proposition below. Some implications of this dependence for the geometry of the family $\mathcal{F}$ are discussed in Remark \ref{moduliX}.

\begin{prop}\cite[Ex. 13.1.1]{har}\label{Prop:: isoX}
Every abstract deformation of $X$ arises from embedded deformations of $C\subset \PP^3$, and the obstructions to deforming $X$ coincide with those to deforming $C$.
\end{prop}

\subsection{Hypersurfaces via subvarieties of codimension two.}\label{construction}
The small resolution $X\to Y$ of a linear determinantal quartic threefold, and its relation to an ACM curve, is a special case of the following construction. 

\medskip\noindent
An $(a\times b)$-matrix $M$ of linear forms in $c$ variables is a $\mathbb{C}$-linear map of vector spaces $V_a\otimes V_b\to V_c$, where $V_a,V_b$ and $V_c$ have dimensions $a,b$ and $c$, respectively. Since there is an isomorphism   $$Hom(V_a\otimes V_b, V_c)\cong Hom(V_a,Hom(V_b, V_c)),$$ then $M$ can be thought of as a $(c\times b)$-matrix $N$ of linear forms in $a$ variables. Moreover, $M$ and $N$ have the same rank.

\medskip\noindent
Now suppose that $b=c$ and $a=c+1$, and consider
$$Y:=\{\lambda\in \PP^{c}\ | \ det(N_{\lambda})=0\}\subset \PP^{c},$$ where $N_\lambda$ is the evaluation of $N$ in $\lambda$. In this case we have an incidence variety $X$ defined as follows: 
   \begin{equation*}
\xymatrix{ X\ar[d]&\!\!\!\!\!\!\!\!\!\!\!\!\!=\{(p,\lambda)\in \mathbb{P}^{c-1} \times Y  |\ \mbox{line representing\ }p\subset ker(N_{\lambda})\}\ar[r] & Y\subset \PP^{c} \\
 \PP^{c-1} & &} 
\end{equation*}
If $c=3$, then $Y\subset \PP^3$ is a cubic surface, which is the anticanonical model of $X$, the blow-up of six points on $\PP^2$. The generic $Y$ arises from six points in general position, and if these points specialize to lie on a conic, then the linear system $|-K_X|$ contracts that conic to a double point of $Y$.

\medskip\noindent
The case $c=4$ is what we study in this paper. Here, we show that the generic linear determinantal quartic $Y$ arises as the anticanonical model of the blowup $X=Bl_C\PP^3$, where $C$ is a smooth and irreducible curve of degree and genus $(d,g)=(10,11)$, which is general in its component of the Hilbert scheme (Section \ref{sec::3}). If $C$ specializes to a curve that lies in a smooth cubic surface $S$, then the quartic $Y$ also specializes and acquires an elliptic singularity. In this case, the anticanonical linear system $|-K_X|$, which is base-point free for the general $C$, yields a \textit{rational} contraction of the cubic surface $S$. The curve $C$ can also specialize so that $\OO_C(2)=\omega_C$. In this case, $Y$ specializes to a double quadric with twenty nodes.

\medskip
\subsection{Determinantal quartics}

\noindent
There are five distinct families of determinantal quartics, assuming we restrict to matrices with entries that are homogeneous polynomials of positive degree. The degrees of the corresponding matrices are as follows:
\begin{align*}
    \mathcal{F}:\left(\begin{array}{cccc}1&1&1&1\\1&1&1&1\\1&1&1&1\\1&1&1&1\end{array}\right)\qquad\mathcal{F}_1:\left(\begin{array}{cc}1&3\\1&3\end{array}\right)\qquad\mathcal{F}_2:\left(\begin{array}{cc}1&2\\2&3\end{array}\right)\\
    \mathcal{F}_3:\left(\begin{array}{ccc}1&1&2\\1&1&2\\1&1&2\end{array}\right)\qquad\mathcal{F}_4:\left(\begin{array}{cc}2&2\\2&2\end{array}\right).
\end{align*}
In analogy to \cite[Lem. 4.2]{LLV}, each of these families is characterized by the following property.

\begin{lem}\label{5families}
There are five irreducible families of quartic threefolds, each distinguished by a general element containing the following type of surface:
\begin{itemize}
    \item $\mathcal{F}$ : a degree 6 surface $S$ whose ideal sheaf $\mathcal{I}_S$ admits the resolution 
    $$ 0 \to \mathcal{O}_{\PP^4}(-4)^3 \to \mathcal{O}_{\PP^4}(-3)^4 \to \mathcal{I}_S\to 0. $$
    \item $\mathcal{F}_1$: a 2-plane.
\item $\mathcal{F}_2$ : a quadric.
\item $\mathcal{F}_3$ : a cubic scroll.
\item $\mathcal{F}_4$ : the complete intersection of two quadrics. 
 \end{itemize}
    \end{lem}

\medskip\noindent
The family of linear determinantal quartics
$$\mathcal{F}=\overline{\{Y=\{det \ A=0\}\subset \PP^4 \ | \ A \mbox{ is a }4\times 4 \mbox{ matrix of linear forms} \}}\subset \PP^{69}$$
is listed in the first bullet, and its characterization will be proved in Lemma \ref{lem::bordiga}.

\medskip\noindent
A general element of $\mathcal{F}_1,\mathcal{F}_2$ and $\mathcal{F}_3$ is nodal and has defect $\rho>0$, but it is irrational \cite{cheltsov, beauville3}. The defect $\rho$ is given by the relation $Cl(Y)\cong\PIC(Y)\oplus \mathbb{Z}^{\raisebox{2pt}{$\scriptstyle\rho$}}$, and if $X$ is a hypersurface in $\PP^4$ with terminal singularities, then it is $\mathbb{Q}$-factorial if and only if it is factorial \cite[pp. 2]{corti&mella}.

\medskip
\section{curves of degree 10 and genus 11}\label{Sec2}

\noindent
This paper analyzes the birational geometry of the blow-up $X:=Bl_C(\PP^3)$, where $C$ is a smooth curve of degree 10 and genus 11. This section describes some geometry these curves can have. 

\medskip\noindent The Hilbert scheme of such curves, denoted by $\mathrm{Hilb}_{10,11}$, has at least three irreducible components parametrizing locally Cohen-Macaulay subschemes. Nonetheless, there is a unique irreducible component of $\mathrm{Hilb}_{10,11}$ whose generic point parametrizes a smooth arithmetically Cohen-Macaulay (ACM) curve. This is a direct consequence of \cite[Thm. 2]{ell}. The ideal sheaf of such an ACM curve $C$ admits the following minimal free resolution:
\begin{equation}\label{eq::genMFR}
    0\to\OOT(-5)^4\to\OOT(-4)^5\to\mathcal{I}_C\to0.
\end{equation}

\medskip\noindent
Recall that a curve $C\subset \PP^3$ is ACM if the restriction map $H^0(\OO_{\PP^3}(k))\to H^0(\mathcal{O}_C(k))$ is onto for any $k$.

\begin{prop}
    There is a unique irreducible component of $\mathrm{Hilb}_{10,11}$, denoted by $\mathcal{H}_{10,11}$, whose general element parametrizes a smooth and irreducible ACM curve.
\end{prop}

\medskip\noindent 
Let $\mathcal{H}_{10,11}^\circ$ denote the open set of $\mathcal{H}_{10,11}$ parametrizing smooth and irreducible ACM curves.  An important invariant of a curve $C\subset\PP^3$, called the \textit{Hartshorne-Rao module}, provides a tool to study the complement of $\mathcal{H}_{10,11}^{\circ}$. This is the second time that this module has been used for this purpose \cite{vite}.

\begin{definition}
    Let $C\subset\PP^3$ be a curve. Its \textit{Hartshorne-Rao module} is defined as
    $$HR(C):=\bigoplus_{k\in\mathbb{Z}}H^1(\mathcal{I}_C(k)).$$
\end{definition}

\medskip\noindent 
Note that $HR(C)$ is a graded module of finite length. A curve $C$ is ACM if and only if $HR(C)=0$. Let us focus on two families of curves in $\mathcal{H}_{10,11}$ for which $HR(C) \ne 0$.  

\begin{definition}\label{def::typesOfCurves}
    The family of smooth curves $C\in\mathcal{H}_{10,11}$ with $\OO_C(2)=\omega_C$ is denoted $D_1^\circ$; we call these curves \textit{semicanonical}. Similarly, the family of smooth curves contained in a smooth cubic surface is denoted by $D_2^\circ$. 
\end{definition}

\medskip\noindent
Let $D_i$ stands for the closure of $D_i^\circ$ in $\mathcal{H}_{10,11}$. If $C_1\in D_1^\circ$ and $C_2\in D_2^\circ$ are general, then $H^1(\PP^3,\mathcal{I}_{C_k}(m))=0$, except in the following two cases:
\begin{align}\label{eq::rank1HR}
    h^1(\PP^3,\mathcal{I}_{C_1}(2))=1,\qquad h^1(\PP^3,\mathcal{I}_{C_2}(3))=1.
\end{align}

 \medskip\noindent
 In other words, $HR(C_1)\cong\mathbb{K}(-2)$ and $HR(C_2)\cong\mathbb{K}(-3)$, where $\mathbb{K}(-a)$ denotes the unique graded module isomorphic to $\mathbb{K}$ in degree $a$, and zero in every other degree. Smooth ACM curves together with the curves from Definition \ref{def::typesOfCurves} describe \textit{all} the smooth irreducible curves of degree 10 and genus 11 in $\PP^3$. To establish this, we examine the residual curve to a general $C\in \mathcal{H}_{10,11}$ in a complete intersection of two quartics. The Hartshorne-Rao modules of $C$ and its residual curve $C'$ satisfy the following relation.

\begin{prop}\label{prop::formulaEstrellita}\cite[Prop 1.13]{vite}
    Let $C\subset\PP^3$ be a locally Cohen-Macaulay curve of degree 10 and genus 11, and let $C'$ be the residual curve in a complete intersection of two surfaces of degree 4, at leas one of which is smooth. Then
\begin{equation*}
    h^{0}(\PP^{3},\mathcal{I}_{C}(m))=h^{1}(\PP^{3},\mathcal{I}_{C^{\prime}}(4-m))-h^{0}(\PP^{3},\mathcal{I}_{C^{\prime}}(4-m))+\frac{(m-3)(m-2)(m+2)}{6}
    \end{equation*}
    for all $m$.
\end{prop}
\medskip\noindent 
The residual curve $C'$ has degree 6 and genus 3, and there are only five possibilities for $HR(C')$.
\begin{theorem}\label{thm::HR63}\cite[Thm. 1]{am1} 
    Let $C'\subset\PP^3$ be a locally Cohen-Macaulay curve of degree 6 and genus 3. Then $HR(C')$ is isomorphic to one of the following graded modules: 
    \begin{enumerate}
        \item The zero module,
        \item $\mathbb{K}(-2)$,
        \item $\mathbb{K}(-3)$,
        \item  $\mathbb{K}[x,y,z,w](-1)/\langle x,y,z,w^{3}\rangle$,
        \item $\mathbb{K}[x,y,z,w](-2)/\langle x,y,F,G\rangle$ with $F,G$ polynomials of degree 3 and 7 respectively.
    \end{enumerate}
\end{theorem}

\medskip\noindent
We now prove that only the first three of these five options can occur as the Hartshorne-Rao module of residual curves to a smooth $C\in \mathcal{H}_{10,11}$, each corresponding to $C$ lying in $\mathcal{H}_{10,11}^\circ$, $D_1^\circ$ or $D_2^\circ$.

\begin{theorem}\label{prop::smoothCurves}
    Let $C$ be a smooth irreducible curve of degree $10$ and genus $11$. Then $C$ lies in the union $\mathcal{H}_{10,11}^{\circ}\cup D_1^{\circ}\cup D_2^{\circ}$.
\end{theorem}
\begin{proof}
    Let $C\in\mathrm{Hilb}_{10,11}$ be smooth and irreducible curve, and let $C'$ be the residual curve to $C$ in a complete intersection of two quartic surfaces, at least one of which is smooth. Suppose that $HR(C')$ is as in item (4) or (5) in Theorem \ref{thm::HR63} above. By Proposition \ref{prop::formulaEstrellita} with $m=2$, it follows that
    $$h^{0}(\PP^{3},\mathcal{I}_{C}(2))\leq-4 \quad \mbox{or}\quad \geq 2,$$ depending on wether $HR(C')$ is as in case (4) or (5), respectively. Note both cases are impossible.

    Therefore, $HR(C')$ falls into one of cases (1), (2) or (3) in Theorem \ref{thm::HR63}. By \cite[Thm. 4.1]{vite}, this is equivalent to $C$ belonging to $\mathcal{H}_{10,11}^\circ, D_1^\circ$ or $D_2^\circ$, respectively.\\
\end{proof}

\medskip\noindent
The following proposition guarantees that the specializations in Corollary \ref{mockcorollary::introCor} can be taken over an irreducible base.

\begin{prop}\cite{vite}\label{toolvite}
$D_1$ and $D_2$ are contained in $\mathcal{H}_{10,11}$ and each forms an irreducible divisor. 
\end{prop}

\medskip\noindent
The following lemma follows from \cite[Cor. 4.3 \& Cor. 4.6]{vite}. 
\begin{lem}\label{lem::genInDegX}
    Let $C\in D_1$ (respectively, $C\in D_2$) be general. Then $\mathcal{I}_C$ is generated in degree 4 (respectively, degree 5).
\end{lem}

\begin{rmk} Let $\mathcal{H}_{3,6}$ denote the unique component of smooth ACM curves of degree 6 and genus 3. In this component, the locus of ACM curves is a dense open subset. Its complement consists of two irreducible divisors, corresponding to cases (2) and (3) in Theorem \ref{thm::HR63}. It turns out that each of these divisors is extremal in the effective cone $\EFF \ \! \mathcal{H}_{3,6}$ \cite[Thm. 2.4]{vite}. The analogous divisors in $\mathcal{H}_{10,11}$ are described by Definition \ref{def::typesOfCurves}.
\end{rmk}

\subsection{An explicit curve in $D_1$}
In this subsection we describe an explicit curve $C\in D_1$ using \texttt{Macaulay2} \cite{m2}. We construct $C$ by liaison with a curve $C_0$ which is of bidegree $(2, 4)$, and the sum of lines, in the quadric $\{xw-yz=0\}$. Consider the complete intersection of two quartic surfaces $F_1,F_2$, such that $F_1\cap F_2=C\cup C_0$.  We now verify that $C$ is smooth, irreducible, has the correct degree and genus and it is semicanonical. Moreover, we check that $\dim |11H-3E|=3$.

\begin{verbatim}
    P3 = QQ[x, y, z, w];
    C0 = ideal(y*z-x*w, x^2*z*w-x*y*w^2, x^3*w-x*y^2*w, x^3*y-x*y^3);
    F1 = (x*z+x*w+2*y*w+w^2)*(x*w-y*z) + x^2*z*w-x*y*w^2;
    F2 = (z^2+w^2)*(x*w-y*z) + x^3*y-x*y^3;
    C = saturate(ideal(F1, F2), C0);
    
    print(codim C, degree C, genus C)	-- Invariants
    codim trim ideal singularLocus C == 4 	-- Smoothness
    rank HH^1(sheaf module C) == 0		-- Irreducibility
    rank HH^0((sheaf(P3^1/C))(2)) == 11	-- C is semicanonical
    C3 = saturate(C^3);
    rank HH^0((sheaf module C3)(11)) - 1 	-- Compute dim |11H-3E|
\end{verbatim}

\medskip\noindent
As a corollary of these explicit computations, and the fact that $D_1\subset\mathcal{H}_{10,11}$ by Proposition \ref{toolvite}, we obtain the following.
\begin{cor}\label{cor::explicitSemicanonical}
    Let $C$ be a general element in $\mathcal{H}_{10,11}$ or in $D_1$, and let $X=Bl_C(\PP^3).$ Then $\dim |11H-3E|\leq3.$
\end{cor}

\medskip\noindent
Lemma \ref{lemma11he} below will show that this, in fact, is an equality. This code will be used again in Lemma \ref{lem::smoothRam}.

\medskip
\section{Mori program of $Bl_C \PP^3$ with $C$ smooth ACM}\label{sec::3}
\noindent
This section runs the Minimal Model Program on $X=Bl_C\ \PP^3$, when $C$ is a smooth ACM curve in $\mathcal{H}_{10,11}$. Sections \ref{S5} and \ref{S4} will focus on the cases $C\in D_1$ and $C\in D_2$, respectively.

\medskip\noindent    
We begin by describing the stable base locus decomposition (SBLD) of $X$. Here, the divisor classes are written in terms of the basis $N^1(X)=\langle H, E\rangle$ discussed in Section \ref{sec::2}, where $H$ and $E$ stand for the pullback of $\mathcal{O}_{\PP^3}(1)$ and the exceptional divisor, respectively.

\medskip\noindent By Cayley's formula \cite{barz}, the generic curve $C\in \mathcal{H}_{10,11}$ carries exactly twenty 4-secant lines. We will denote the union of these lines by $L$.

\begin{notat}\label{notat}
    Given non-zero classes $A,B\in N^1(X)$, we denote by $[A, B)\subset N^1(X)$ the convex cones of linear combinations $aA+bB$ where $a> 0$ and $b\geq 0$. Variants of this definition, such as $[A, B]$, $(A, B]$,  $(A, B)$, are defined similarly.
\end{notat}

\medskip\noindent        
We begin by describing the nef cone of $X$.

\begin{prop}\label{prop::nefConeGeneral}
  Let $C$ be a smooth ACM curve. The nef cone of $X$ is $$\mathrm{Nef(X)}=[-K_X,H].$$
\end{prop}
\begin{proof}
    Observe that $H$ is base-point free as it is the pullback of a very ample class. By the resolution \eqref{eq::genMFR}, the ideal sheaf $\mathcal{I}_C$ is generated in degree four and, consequently, $-K_X=4H-E$ is base-point free. Therefore, both classes are nef.
    
    Moreover, these classes are not ample: $H$ contracts the exceptional divisor $E$, and $-K_X$ contracts any 4-secant line to $C$.\\
\end{proof}
    
\begin{lem}\label{lem::excLocus}
    Let $C\in \mathcal{H}_{10,11}$ be either smooth ACM or general in the family $D_1$. Let
    $$f:X\xrightarrow{|-K_X|} Y$$
    be the anticanonical morphism. Then the exceptional locus of $f$ is $Exc(f)=L$, the union of the 4-secant lines to $C$.
\end{lem}
\begin{proof} 
    In both cases, the anticanonical map $f$ is base-point free (by the resolution \eqref{eq::genMFR} or Lemma \ref{lem::genInDegX}, respectively) and the 4-secant lines are contracted. Thus, we need to argue that no more curves are contracted.

    Given $p\in X$, denote by $\mathcal{D}_p\subset|-K_X|$ the sublinear system of divisors whose support contains $p$. Since $f$ is base-point free, the codimension of $\mathcal{D}_p$ in $|-K_X|$ is 1 for every $p\in X$.

    Let $\gamma\subset X$ be an irreducible and reduced curve contracted by $f$. Let us argue that $\gamma$ must be a 4-secant line. We know that
    $-K_X\cdot\gamma=0$. In particular, if $p\in\gamma$ is any point, then an element $F\in|-K_X|$ contains $p$ if and only if it contains $\gamma$. Thus, the locus $\{F\in|-K_X|:\gamma\subset F\}$ is a 3-plane in $|-K_X|\cong\PP^4$, and we can choose two different elements $F,F'\supset\gamma$ in it.
    
    The fact that $\gamma$ is contained in the complete intersection $F\cap F'$ implies that the class $\gamma':=(-K_X)^2-\gamma$ is effective. Since $H$ is nef, then
    $$H\cdot\gamma=H\cdot(-K_X)^2-H\cdot\gamma'\leq H\cdot(-K_X)^2=6.$$
    On the other hand, $H\cdot\gamma\neq0$, since otherwise $E\cdot\gamma=4(H\cdot\gamma)-(-K_X)\cdot\gamma=0$, which is absurd. Thus, $H\cdot\gamma>0$. If $H\cdot\gamma=1$, then $\gamma$ is a 4-secant line to $C$. We are left to prove that $2\leq H\cdot\gamma\leq6$ is impossible.

    Denote $d:=H\cdot\gamma$. Observe from the equality $(4H-E)\cdot\gamma=0$ that $(4H-E)\cdot\gamma'=(4H-E)^3=4$. From this we deduce the identities
    $$H\cdot\gamma'=6-d\qquad E\cdot\gamma'=4(5-d).$$
    Fix a smooth surface $F\in|-K_X|$ containing $\gamma$ and $\gamma'$. In $F$, we have that $\gamma+\gamma'=(4H-E)|_F$. Therefore
    $$\gamma^2-(\gamma')^2=(\gamma+\gamma')\cdot(\gamma-\gamma')=-(4H-E)\cdot\gamma'=-4.$$
    Since $F$ is a K3 surface, this leads to the following identity involving the arithmetic genus of $\gamma$ and $\gamma'$:
    $$p_a(\gamma)=p_a(\gamma')-2.$$
    If $d\geq3$, then $deg(\gamma')=6-d\leq3$ and $p_a(\gamma)=p_a(\gamma')-2<0$, which is a contradiction since $\gamma$ is reduced and irreducible. If $d=2$ then $p_a(\gamma)=0$. Therefore, $p_a(\gamma')=2$. 
    
    Let us argue that $\gamma'$ is contained in a quadric surface. First, we prove that $\gamma'$ is reduced. Otherwise, we can write its divisor class in $F$ as $\gamma'=2c,2\ell_1+2\ell_2$ or $2\ell_1+\ell_2+\ell_3$, where $c$ and $\ell_i$ (i=1,2,3) denote the classes of an irreducible conic and a line, respectively. In $F$ we have
    \begin{align*}
        (2c)^2&=-8,\\
        (2\ell_1+2\ell_2)&=-16+8\ell_1\cdot\ell_2=-16\textrm{ or }-8,\\
        (2\ell_1+\ell_2+\ell_3)^2&=-12+4\ell_1\cdot\ell_2+4\ell_1\cdot\ell_3+2\ell_2\cdot\ell_3=-12,-8,-6,-4,\textrm{ or }-2
    \end{align*}
    depending on wether or not the lines intersect each other. In any of these cases, the possible self-intersection numbers are different from $(\gamma')^2=2.$ Consequently, $\gamma'$ must be reduced.

    Observe that $\gamma'$ is connected: otherwise $\gamma'=\gamma'_1\sqcup\gamma'_2$ with $p_a(\gamma'_i)\leq 3$. This would lead to $p_a(\gamma')=p_a(\gamma'_1)+p_a(\gamma'_2)-1\leq1$, which is a contradiction. 

    Finally, the decomposition into irreducible components of $\gamma'$ must fall in (at least) one of the following cases:
    \begin{enumerate}[a)]
        \item $\gamma'$ is irreducible,
        \item $\gamma'=\ell\cup q$, where $\ell$ is a line, $q$ is a degree 3 irreducible curve, and $\ell\cap q\neq\emptyset$,
        \item $\gamma'=c_1\cup c_2$ is a union of irreducible conics with $c_1\cap c_2\neq\emptyset$,
        \item $\gamma=c\cup \ell_1\cup\ell_2$ with $c\cap\ell_1,\ell_1\cap\ell_2\neq\emptyset$,
        \item $\gamma'=c\cup\ell_1\cup\ell_2$ with $c\cap\ell_1,c\cap\ell_2\neq\emptyset$,
        \item $\gamma'=\ell_1\cup\ell_2\cup\ell_3\cup\ell_4$ with $\ell_i\cap\ell_{i+1}\neq\emptyset$ for $i=1,2,3$, or
        \item $\gamma'=\ell_1\cup\ell_2\cup\ell_3\cup\ell_4$ with $\ell_i\cap\ell_4\neq\emptyset$ for $i=1,2,3$.
    \end{enumerate}
    In each of these cases, we can choose nine points on $\gamma'$ where each degree $k$ component contains at least $2k+1$ of them. B zout's theorem now guarantees that the quadric through such 9 points contains each irreducible component of $\gamma'$.
    
    Now, a smooth quadric has no genus 2 curves of degree 4; a degree 4 curve in a conic cone has genus 1; and a degree 4 curve in a double plane is either a plane quartic (genus 3), a double conic (genus 1), or a line plus a triple line (genus 3). This is a contradiction.

\end{proof}

\noindent
\begin{rmk}
    If $C$ is ACM, we can prove Lemma \ref{lem::excLocus} as follows: all 4-secant lines to $C$ are contracted into singular points of $Y$. By the proof in Proposition \ref{lem4IsInjective} below, $Y$ is a linear determinantal quartic. Thus $Y$ is singular along isolated points by \cite[Prop. 1.3.6, (b)] {oslo} (therein called points of rank 2), and the inverse image of such points under $f$ is precisely the union of the $4$-secant lines to $C$ \cite[Prop. 1.3.6, (c)]{oslo}. This proves that $Exc(f)=L$, provided that $f$ does not contract any divisor.
    
    To see that $f$ is not a divisorial contraction, observe that $h^{2,1}(X)=g(C)=11$ means that the twenty singular points of $Y$ fail to impose independent conditions on cubic hypersurfaces \cite{cynk}. This implies that $Y$ is not $\mathbb{Q}$-factorial, thus $|-K_X|$ induces a small contraction.
   \end{rmk}

\medskip\noindent
We now explore the anticanonical model of $X$.

\begin{prop}\label{lem4IsInjective}
Let $C\in \mathcal{H}_{10,11}$ be smooth and ACM. The anticanonical linear series of $X=Bl_C\PP^3$ yields a birational morphism
$$f:X\xrightarrow{|-K_X|} Y\subset\PP^4$$ onto a quartic 3-fold, $Y$. If $C$ is general, $Y$ is determinantal with 20  singular points.
\end{prop}

\begin{proof}
    First we prove that $f$ is generically injective. Fix a smooth member $F\in|-K_X|$ and a general point $p\in F$. If $f(p)=f(q)$ for some other point $q\in X$, then by definition $q\in F$. We will show that this implies $p=q$.
    
    Consider a general member $F'\in|-K_X|$. The images in $\PP^3$ of $F$ and $F'$ are two smooth quartic surfaces containing $C$, and the residual curve $C'$ is a smooth ACM curve of degree 6 and genus 3, which is not hyperelliptic because it does not lie on a quadric surface, see \cite[Thm. 4.1]{vite}.

    This means that the restriction to $F\subset X$ of the effective divisor $F'$ is the divisor of the strict transform of $C'$, which we will also call $C'\subset F$. By \cite[Prop. VIII.13 (iv)]{Beau}, the complete linear system of $C'$ defines a birational morphism $F\to|C|$.

    Now consider the restriction map $H^0(X,-K_X)\to H^0(F,C')$ appearing in the exact cohomology sequence of
    $$0\to\OO_X\to\OO_X(-K_X)\to\OO_F(C')\to0.$$
    Since $h^0(X,\OO_X)=1$ and $h^1(X,\OO_X)=0$, this restriction map is surjetive. That is, the complete linear system $|C|$ in $F$ is cut-out by restrictions of divisors in $|-K_X|$. This implies that the restriction $f|_F:F\to|-K_X|\cong\PP^4$ is birational onto its image. In particular, $f(p)=f(q)$ implies that $p=q$, since $p\in F$ is general.

    Since smooth members of $|-K_X|$ cover a dense open set of $X$, this argument proves that $f(p)=f(q)$ implies $p=q$ for the general $p\in X$, so that $f:X\to|-K_X|\cong\PP^4$ is birational onto its image $Y$. Lemma \ref{lem::excLocus} implies that $Y$ acquires exactly 20 nodes, and that the degree of $Y$ is $(-K_X)^3=4$.
    
    Moreover, the image of a general member in the linear system $|H|$ under $f$, where $H$ is the pull-back of the a hyperplane section in $\PP^3$, is a \textit{Bordiga surface} $S$ in $\PP^4$ of degree 6. This surface satisfies that $\mathcal{I}_S$ has a minimal free resolution as follows
    \begin{equation}\label{eq::bordigaMFR}
        0\to \mathcal{O}_{\PP^4}(-4)^{3}\overset{N}{\to} \mathcal{O}_{\PP^4}(-3)^{4}\to \mathcal{I}_{S}\to 0.
    \end{equation}
    We conclude by the following lemma.
\end{proof}

\medskip\noindent
A Bordiga surface $S\subset \PP^4$ is defined as a smooth surface whose ideal sheaf $\mathcal{I}_S$ admits a minimal free resolution like (\ref{eq::bordigaMFR}). Containing one of these surfaces has the following implication.

\begin{lem}\label{lem::bordiga}
    If $S\subset\PP^4$ is a Bordiga surface and $Y$ is a quartic hypersurface containing $S$, then $Y$ is linear determinantal.
\end{lem}
\begin{proof}
    The map in the resolution \eqref{eq::bordigaMFR} admits a representation as a $(4\times 3)$-matrix $N=(n_{ij})$, where the entries are linear forms in variables $x,y,z,w$. A set of minimal generators $\{F_1,\ldots,F_4\}$ of $\mathcal{I}_S$ is obtained by  defining $F_i$ as the ($3\times 3$)-minor of $N$ resulting from removing the $i$-th row. If $Y$ is the hypersurface given by a degree 4 polynomial $F$, then $S\subset Y$ implies that $F$ can be written as $\ell_1\cdot F_1+\ldots+\ell_4\cdot F_4$ for some linear forms $\ell_i$. In other words, the following identity holds:
    $$F=det\left(\begin{array}{cccc}
    \ell_1 & n_{11} & n_{12} & n_{13}\\
    -\ell_2 & n_{21} & n_{22} & n_{23}\\
    \ell_3 & n_{31} & n_{32} & n_{33}\\
    -\ell_4 & n_{41} & n_{42} & n_{43}
    \end{array}\right),$$
    which means that $Y$ is linear determinantal as claimed.
\end{proof}

\begin{cor}
    $X$ is a Mori dream space.
\end{cor}
\begin{proof}
    The class $-K_X$ is nef and big and therefore $X$ is weak Fano variety. Since $X$ is smooth, by \cite{mass}*{Prop. 2.6} is log Fano. The corollary now follows from \cite{BCHM}*{Cor. 1.3.2}\\
\end{proof}

\medskip\noindent
The following Proposition, via a computation of a tangent space, shows that the family $\mathcal{F}$ is the largest irreducible family that contains linear determinantal quartics.

\begin{prop}\label{familiaF}
The family $\mathcal{F}$ of linear determinantal quartic threefolds in $\PP^4$ is irreducible and generically smooth. Moreover, a general element in $\mathcal{F}$ is the anticanonical model of $X=Bl_C\PP^3$ for an ACM curve $C\in \mathcal{H}_{10,11}$.
\end{prop}
\begin{proof} 
The irreducibility of $\mathcal{F}$ is clear (see \cite[Prop. 1.1]{LLV}).

In \cite[Thm. 1.3]{reivis} it is proved that $\mathcal{F}$ has dimension 49. Since a general $Y$ is singular at twenty points $\Gamma:=Sing(Y)$, the Zariski tangent space $T_Y\mathcal{F}$ at a point $Y$ is contained in the space of embedded first-order deformations of $Y$ not smoothing out $\Gamma$. This space $T_Y\mathcal{F}$ is isomorphic to $H^0(Y,\mathcal{I}_{\Gamma/Y}(4))$, where $\mathcal{I}_{\Gamma/Y}$ denotes the ideal sheaf of $\Gamma\subset Y$ \cite[pag.255]{sernesi}. Therefore, $\mathcal{F}$ is smooth at $Y$ provided $\Gamma$ imposes independent conditions on sections of $\OO_Y(4)$. This can be verified computationally. For example, if $C$ is given by the maximal minors of the matrix
\begin{equation*}
    \left(\begin{array}{cccc}
        y& 0& x+y+z+w& y+z+w\\
        y& 0& z& w\\
        y+z& x+y& 0& y+z+w\\
        x& x+w& x& x+z\\
        0& x+y+z& z+w& z+w
    \end{array}\right)
\end{equation*}
inside the projective space $\PP^3$ with coordinates $[x:y:z:w]$, then the singular locus $\Gamma$ of $Y$ consists of 20 double points whose ideal has Betti table
\begin{equation*}
    \begin{array}{c|cccc}
        & 1& 2& 3& 4\\\hline
        2& 16& 30& 16& -\\
        3& -& -& -& -\\
        4& -& -& -& 1
    \end{array}
\end{equation*}
It is now easy to check that $\Gamma$ imposes independent conditions in sections of $\OO_Y(4H)$.

\end{proof}

\begin{rmk}\label{max1}
    The nodes of a general $Y\in \mathcal{F}$ are in a special position. Indeed, the dimension of the intermediate Jacobian $IJ(X)$ is 11, as it coincides with the Jacobian of a smooth curve $C\in \mathcal{H}_{10,11}$. This implies that the nodes of $Y$ fail to impose independent conditions on cubics \cite{cynk}.  Note that Proposition \ref{familiaF} rules out the possibility of $Y$ arising as a specialization, over an irreducible base, of a 20-nodal quartic with nodes in general position: such a quartic would be birationally rigid and, in particular, non-rational \cite{mell}. Hence, $\mathcal{F}$ is the largest family of quartics with twenty nodes and defect $\rho=1$. 
    \end{rmk}

\medskip\noindent
\begin{rmk}\label{moduliX}
Thus far the paper has focused on extrinsic geometry. Intrinsically, we can say the following. The GIT quotient $\mathcal{M}:=\mathcal{F}//\PP GL(5,\mathbb{C})$ is generically smooth of dimension $25$. 
The intrinsic information of a curve $C\in \mathcal{H}_{10,11}$ is encoded in the isomorphism class of $C$ and a complete linear series of degree 10 and dimension 3 on $C$, denoted $g^3_{10}$. Thus, the moduli space of pairs $$\mathcal{W}:=\{(C,g_{10}^3) \ |\ C \mbox{ carries the }g^3_{10}\}$$ admits a map $\xi:  \mathcal{W}\dashedrightarrow \mathcal{M}$, by considering the blow-up $X=Bl_C\PP^3$ and taking its anticanonical model. 
 Observe that $\xi$ is generically finite for dimensional reasons and Proposition \ref{Prop:: isoX}. Corollary B provides a step towards showing that $\xi$ is well-defined throughout $\mathcal{W}$.

More precisely, let $ \mathcal{M}_{11,10}^3\subset \mathcal{M}_{11}$ be the locus of curves of genus 11 that admit a $g^3_{10}$. Projection onto the first coordinate yields the diagram
\begin{equation*}
\xymatrix{ \mathcal{W}\ar[d]_{\pi} \ar@{-->}[r]_{\xi}&  \mathcal{M} \\
 \mathcal{M}_{11,10}^3& } 
\end{equation*}

Observe that $\pi$ is onto.
 Since the residual linear series $|K_C-H|$ to a $g^3_{10}$ is also a $g^3_{10}$, then $\pi$ has degree at least 2. Consequently, there exists a ramification divisor $Ram(\pi)\subset \mathcal{W}$.

The fibers of $\pi$ parametrize the $g^3_{10}$'s on a fixed $C$. Two of them yield the blown-ups of $\PP^3$ that appear in the Sarkisov link of a general $X$. Note that the image of $Ram(\pi)$ under $\xi$ contains isomorphisms classes of quartics that arise from semicanonical curves, $i.e.$ double quadrics, as long as they are GIT-semistable. 

We do not know if the quartics that arise from curves contained in cubic surfaces (see Theorem \ref{teorema2}, (4)) are GIT-semistable and hence define a locus in $\mathcal{M}$. It would be interesting to know if $\xi$ is well-defined over $\mathcal{W}$ and constant on the fibers $\pi^{-1}(C)$, hence inducing an isomorphism $\mathcal{M}^3_{11,10}\cong \mathcal{M}$.
\end{rmk}

\medskip\noindent
We now turn our attention to the divisor class $11H-3E$, which spans an extremal ray in the movable cone $\mathrm{Mov}(X)$. 

\begin{lem}\label{lemma11he}
    The linear system $|11H-3E|$ has dimension 3.
\end{lem}
\begin{proof}
    According to Corollary \ref{cor::explicitSemicanonical}, $\dim\ |11H-3E|\le 3$ for a generic curve $C$. Let us show that the opposite inequality also holds. 
    
    First, observe that $-3H+E=K_X+H$,
    where $H$ is nef and big. It follows from the Kawamata-Viehweg vanishing theorem that
    $H^i(X,-H)\cong H^{3-i}(X,-3H+E)=0$
    for every $i<3$. Since $\chi(X,-H)=0$, then $H^i(X,-H)=0$ for all $i$.

    Now consider a general element $F\in|4H-E|$, which we think of as the strict transform of a smooth quartic surface containing $C$. The exact cohomology sequence induced by the following sequence
    $$0\to\OO_X(-H)\to\OO_X(3H-E)\to\OO_F(3H-E)\to 0,$$
    and the vanishing above, imply that
    $$H^i(X,3H-E)\cong H^i(F,3H-E)$$
    for all $i$. Since the general $C$ is not contained in a cubic surface, then these groups are trivial if $i=0$. On the other hand, observe that
    $$H^2(F,3H-E)\cong H^0(F,-3H+E)=0.$$
    Indeed, the restriction of $H$ to the surface $F$ is an ample class and the intersection $(-3H+E)\cdot H$ in $F$ is equal to
    $(4H-E)\cdot(-3H+E)\cdot H=-2<0.$
    Therefore, the restriction of $-3H+E$ to $F$ cannot be effective. Putting all these computations together, along with the fact that $\chi(F,3H-E)=0$ (see Lemma \ref{lem::chi}), implies that $H^i(X,3H-E)=0$ for every $i\ge 0$.
    
    A similar argument, this time using the following exact sequence,
    $$0\to\OO_X(3H-E)\to\OO_X(7H-2E)\to\OO_F(7H-2E)\to0$$
    proves that $H^i(X,7H-2E)=0$ for all $i$.

    Finally, we iterate once more the argument above and consider the sequence
    $$0\to\OO_X(7H-2E)\to\OO_X(11H-3E)\to\OO_F(11H-3E)\to0.$$
    It follows that $H^3(X,11H-3E)=0$ and
    $$H^2(X,11H-3E)\cong H^2(F,11H-3E)\cong H^0(F,-11H+3E)=0,$$
    as we have that
    $(4H-E)\cdot(-11H+3E)\cdot H<0$. We arrive at the inequality
    $$h^0(X,11H-3E)\geq\chi(X,11H-3E)=4.$$
\end{proof}

\begin{prop}\label{SLink1} The generic $Y\in \mathcal{F}$ fits into the following diagram
\begin{equation} \label{diagprop5.8}
    \xymatrix{
        & X \ar@{-->}[rr]^{\sigma} \ar[dr]^{f} \ar[dl]_{\pi} & &X^{+}\ar[dr]^{\pi^{+}} \ar[dl]_{f^+}& \\
        \PP^{3}& & Y& & \PP^3\cong|11H-3E|, 
    } 
\end{equation}
where $\sigma$ is a flop and $\pi^+$ is the blow-up of $\PP^3$ along a generic curve $C^+$ of degree 10 and genus 11.
\end{prop}
\begin{proof}
This proof follows closely ideas from \cite{five}. Since $-K_X$ defines an extremal small contraction $f$, \cite[Thm. 6.14]{KM} implies that the flop $\sigma:X\dashedrightarrow X^+$ exists. Following the notation in \cite{KM}*{Thm. 1.32}, note that $f^{+}$ is exceptional (E) $i.e.$, a divisorial contraction; a conic bundle (C); or a del Pezzo fibration (D).
    
    Suppose that $f^{+}$ is a conic bundle and let $D$ be the pullback of $\OO_{\PP^{2}}(1)$ on $X^{+}$. Then $-K_{X^{+}}\cdot D^{2}=2$. Consider the strict transform of $D$ across the flop $\sigma$, which can be written as $\widetilde{D}=a(-K_{X})+bE$ for some rational numbers $a,b$. Since
    $$(-K_{X})^{2}\cdot E=4d+2-2g=20\qquad \text{and}\qquad (-K_{X})\cdot E^{2}=2g-2=20,$$
    then
    \begin{align*}
        2&= -K_{X^{+}}\cdot D^{2}\\
        &= -K_{X}\cdot \widetilde{D}^{2}\\
        &=-K_{X}\cdot (a(-K_{X})+bE)\\
        &=2(2a^{2}+20ab+10b^{2}).
    \end{align*}
    This equation has no rational solutions. Therefore, this case is not possible.
    
    Now suppose that $f^{+}$ is a del Pezzo fibration and let $D$ be the divisor class of the del Pezzo
    surfaces in the fibration. Then $K_{X^{+}}\cdot D^{2}=0$. Consider the strict transform of $D$ across the flop $\sigma$, $\widetilde{D}=a(-K_{X})+bE$ for some rational numbers $a,b$. Then we similarly have that
    $$0=-K_{X^{+}}\cdot D^{2}=4a^{2}+40ab+20b^{2}=4(a^{2}+10ab+5b^{2}).$$
    Likewise, this equation does not have rational solutions. Then $f^{+}$ is of exceptional type (E) $i.e.$, a divisorial contraction.
    
    Then, following \cite{tow} and its notation, we are in a case of type $E1-Ei$, with $i\in{\{1,2,3,4,5\}}$. According to their tables the cases with $i=2,3,4,5$ are not possible. Therefore, we are in the case $E1-E1$, which implies that $Y^{+}=\PPP$ and $\pi^{+}$ is the blow-up over a smooth curve of degree 10 and genus 11, as claimed.\\
\end{proof}

\medskip\noindent
We now describe the effective cone $\EFF(X)$.

\begin{prop}\label{effectivecone}
    The divisor class $40H-11E$ is effective and rigid. This class spans an extremal ray of the effective cone. In other words, $\mathrm{Eff}(X)=[40H-11E, E]$.
\end{prop}
\begin{proof}
In order to see that $40H-11E$ is effective, consider the diagram \eqref{diagprop5.8} in Proposition \ref{SLink1}. Let $C^+\subset\PP^3$ denote the curve blown-up by $\pi^+$, and let $H^+$ and $E^+$ be the generators of $\PIC(X^+)$. Then, observe that
\begin{align*}
    \sigma^*H^+=11H-3E& \quad \mbox{and} \quad \sigma^*(4H^+-E^+)=4H-E.
\end{align*}
The second identity means that $\sigma$ sends $K_{X^+}$ to $K_X$. It follows from these equations that $\sigma^*(E^+)=40H-11E$ is effective.

\medskip
Let us now show that $40H-11E$ is extremal. To that end, consider the union of the twenty 4-secant lines to $C$, denoted by $L$.
Fix two general divisors $F, F'$ in the 3-dimensional linear system $|11H-3E|$ (Lemma \ref{lemma11he}). Observe that $$F\cap F'=L\cup C',$$ 
where $C'$ is a curve that sweeps out a dense open subset of $X$ as $F,F'$ vary.  Also, $H\cdot C'=H\cdot\left((11H-3E)^2-L\right)=11,$ and similarly
    $E\cdot C'=40$. This implies that $(40H-11E)\cdot C'=0$, which finishes the proof.

\end{proof}

\begin{prop}\label{SBLDgeneric}
    The stable base locus decomposition of the effective cone
of $X$ consists of the following four chambers:
    \begin{enumerate}[(1)]
        \item 
        In the region $(H, E]$, the stable base locus equals the support of $E$.
        \item In the closed cone $[4H-E, H]$, the stable base locus is empty. 
        \item In the region $[11H-3E, 4H-E)$, the stable base locus equals the union $L$ of the twenty 4-secant lines to $C$.
        \item In the region $[40H-11E, 11H-3E)$, the stable base locus is equal to the support $J$ of the rigid divisor class $40H-11E$.
    \end{enumerate}
\end{prop}

\begin{figure}[h] \label{h2} 
    \centering
    \includegraphics[width=0.5\textwidth]{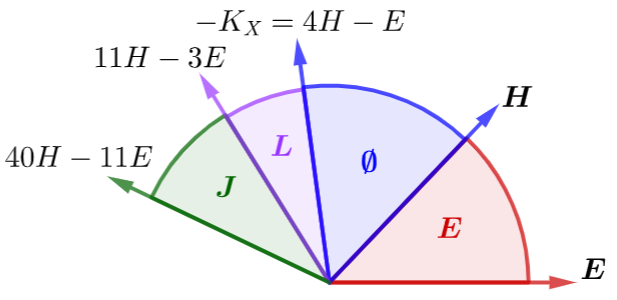}
    \caption{ SBLD of $X=Bl_C\ \PP^3$ for a generic $C$.}
\end{figure}

\begin{proof}\ 
    \begin{enumerate}[(1)]
        \item The class $H$ is base-point free. Then, any linear combination $aH+bE$, with $a,b\ge 0$, has base locus contained in $E$. If $\ell\subset E$ stands for a ruling of $E$, then observe that $(aH+bE)\cdot\ell<0$ as long as $b>0$. This implies that the base locus of this chamber is $E$.

        \item Follows from Proposition \ref{prop::nefConeGeneral}.

        \item Since $K_X$ is nef, the stable base locus $Bs(D)$ is contained in $Bs(11H-3E)$ for any $D\in [11H-3E,-K_X)$. Furthermore the twenty 4-secants to $C$ satisfy $L\subset Bs(11H-3E)$. Then it suffices to show that $Bs(11H-3E)=L$. Looking at Diagram \eqref{diagprop5.8}, we observe that only points in
        $Exc(f)=L$ can be in $Bs(11H-3E)$. 

        \item This follows from the fact that $J=\sigma^*(E^+)$ is contracted by $11H-3E$ (Proposition \ref{effectivecone}). Note that $(40H-11E)\cdot\ell=-4<0$ for any (strict transform of a) 4-secant line to $C$. Therefore, $L\subset J$.
    \end{enumerate}
\end{proof}

\begin{cor}
    The movable cone of $X$ is  $\MOV(X)=[ 11H-3E,H]$.
\end{cor}

\subsection{Birational models of $X$.}
Recall that, given a movable divisor $D$ in $X$, there is an associated \textit{birational model} of $X$ defined as
$$X(D):=\mathrm{Proj}\left( \bigoplus_{m\ge 0}H^0(X,  mD)\right).$$
Let us now describe the different birational models of $X$ induced by each SLBD chamber in $\MOV(X)$. The following theorem is the main result of this section.
\begin{theorem}\label{teorema1}
Let $D$ be an effective divisor in $X$, the blow-up of $\PP^3$ along a general ACM degree 10 and genus 11 curve $C\in\mathcal{H}_{10,11}$. Then, we have the following birational models $X(D)$:
\begin{enumerate}[(1)]
    \item[(0)] $X(H)\cong\PP^3$.
    \item $X(D)\cong X$ for $D\in(-K_X, H)$.

    \item $X(-K_X)\cong Y$ and the map $f:X\rightarrow Y\subset \PP^4$ is a small contraction whose exceptional locus is $\mathrm{Exc}(f)=L$, the union of the twenty 4-secants of $C$.

    \item $X(D)\cong X^{+}$ for $D\in(-K_X,11H-3E)$. The map $f^+: X^{+}\rightarrow Y$ is the flop of $f$, where the flopping locus is $L$, the union of the twenty 4-secants to $C$.

    \item $X(11H-3E)\cong \PP^3$.
\end{enumerate}
\end{theorem}

\noindent
The following diagram summarizes the previous theorem:
\begin{equation}\label{X1DIAGRAM}
\xymatrix @!=1pc{& X\ar[dl]_{\pi} \ar@{-->}[rr]^{\sigma} \ar[dr]^{f}&  &  X^{+} \ar[dl]_{f^+}\ar[dr]^{\pi^+}&\\
\PP^{3} && Y &&  \PP^3} 
\end{equation}

\begin{lem}\cite[Lem. 3.1]{DC}\label{lemma::DC}
Let $f:X\rightarrow Y$ be a birational morphism, where $X$ and $Y$ are normal projective varieties. If $D\subset Y$ is an ample divisor, then 
$$\mathrm{Proj}\left(\bigoplus_{m\ge 0}H^0(X,mf^*D)\right)=Y. $$
\end{lem}

\medskip
\begin{proof}[Proof of Theorem \ref{teorema1}]
Item (0) follows from Lemma \ref{lemma::DC}, while items (1) and (2) follow from the definition of $X$ and Lemma \ref{lem4IsInjective}. In order to show (3), observe that any $D=-aK_X+b(11H-3E)$, where $a\ge 0, b>0$ satisfies that $D\cdot\ell_i<0$ for each 4-secant $\ell_i$. Item (4) follows from Lemma \ref{lemma11he}. This completes the proof.\\ 
\end{proof}

\begin{rmk}\label{AutK3}
    Let $F$ be a general quartic surface containing a general $C\in \mathcal{H}_{10,11}$. It turns out that the automorphism group of $F$ is $Aut(F)\cong \mathbb{Z}$. In \cite[Thm. 2]{caro} it is proved that the generator of this group arises from a birational map $\varphi\in Bir(\PP^3)$ that preserves $F$.  The map $\varphi$ can be understood in terms of the Sarkisov link (\ref{X1DIAGRAM}), and Remark \ref{AutK3II} describes how this situation changes as $C$ specializes.
\end{rmk}

\subsection{The twenty 4-secants} Let us describe the behavior of the twenty 4-secants of the curve $C$ as the Mori program is run on $X$. We use the notation of the diagram (\ref{X1DIAGRAM}) above. We have that $\sigma^*(H^+)=11H-3E$, where $H^+$ denotes the pullback of the hyperplane section under $\pi^+$. By Bertini's Theorem, a general element $F\in |11H-3E|$ is smooth. Note that a smooth surface $F^+$ in the class $H^+\in \PIC(X^+)$ is isomorphic to the blow-up of a plane $\PP^2\subset(\PP^{3})^+$ at the 10 points of intersection with $C^+$.  Furthermore, the surface $B:=f^+(F^+)\subset \PP^4$ is a \textit{Bordiga surface} and it is isomorphic to $\PP^2$ blown-up at 10 points, embedded in $\PP^4$ with the linear system of quartics through the blown-up points.  Notice that the 4-secants to the curve $C^+$ are not contained in $F^+$, but their images under $f^+$ are points on the Bordiga surface $B\subset Y$. 

\medskip\noindent
Now, the morphism $f$ contracts each twenty 4-secant line $\ell_i$ of $C$, and each of them is a (-1)-curve in $F\in |11H-3E|$. In fact, $\ell_i\subset F$ since $F\cdot\ell_i=(11H-3E)\cdot\ell_i=-1$. By the adjuntion formula:
\begin{equation*}
-2=\ell_i^2+K_F\cdot\ell_i=\ell_i^2+(7H-2E)\cdot\ell_i=\ell_i^2-1.
\end{equation*}
Hence,  the image $f(F)=B\subset \PP^4$ is a smooth Bordiga surface and $F$ is isomorphic to the blow-up of $\PP^2$ at 30 points. Let us summarize the previous discussion.

\begin{prop}
If $F\in |11H-3E|$ is smooth, then $F$ is isomorphic to the plane $\pi^+\circ\sigma(F)\subset\PP^3$ blown-up at 30 points: the 10 points of intersection with the curve $C^+$ and the 20 points of intersection with its twenty 4-secant lines.
\end{prop}

\medskip\noindent
The previous 30 points are in special position on the plane. To see this, let us write down the restriction of a class $dH-kE$ in $X$ to $F$ in terms of its description as a blow-up. Let us denote by $\ell$ the pullback to $F$ of $\OO_{\PP^2}(1)$ under the blow-down morphism. Let $s_1,\ldots, s_{20}$ be the classes of the 4-secant lines inside $F$, and $e_1,\ldots,e_{10}$ be the other 10 exceptional divisors. For convenience, we write
$$s:=\sum_{i=1}^{20}s_i,\qquad \mbox{and}\qquad
e:=\sum_{i=1}^{10}e_i.$$

\noindent
Observe the canonical class $K_F$ can be readily computed as follows
$$(7H-2E)|_F=K_F=-3\ell+e+s.$$
On the other hand, the linear series of the restriction of the class $4H-E$ to $F$ yields the morphism into the Bordiga surface inside $\PP^4$. This means that it contracts the exceptional divisor class $s$, and we get
$$(4H-E)|_F=4\ell-e.$$
We summarize this discussion in the following Lemma.
\begin{lem}
    If $F\in|11H-3E|$ is general then $$H|_F=11\ell-s-3e,\quad  E|_F=40\ell-4s-11e.$$
    \end{lem}

\medskip\noindent
    As a corollary of the previous Lemma, we have that the 30 blown-up points to obtain $F$ are in special position. For example, the class $H|_F=11\ell-3s-e$ is not effective for $30$ generic points. Furthermore, the curve $\tilde{C}=E|_F$ satisfies $h^0(S,\mathcal{O}_F(\tilde{C}))=1$, and by the adjunction formula, $\tilde{C}$ has genus 71. Since $\pi$ restricted to $\tilde{C}$ is a $3:1$ ramified cover of $C$, the Riemann-Hurwitz formula predicts 80 ramification points, which are exactly the intersections of $C$ with its twenty 4-secant lines $L$.

\begin{figure}[h]
\centering
\includegraphics[scale=.39]{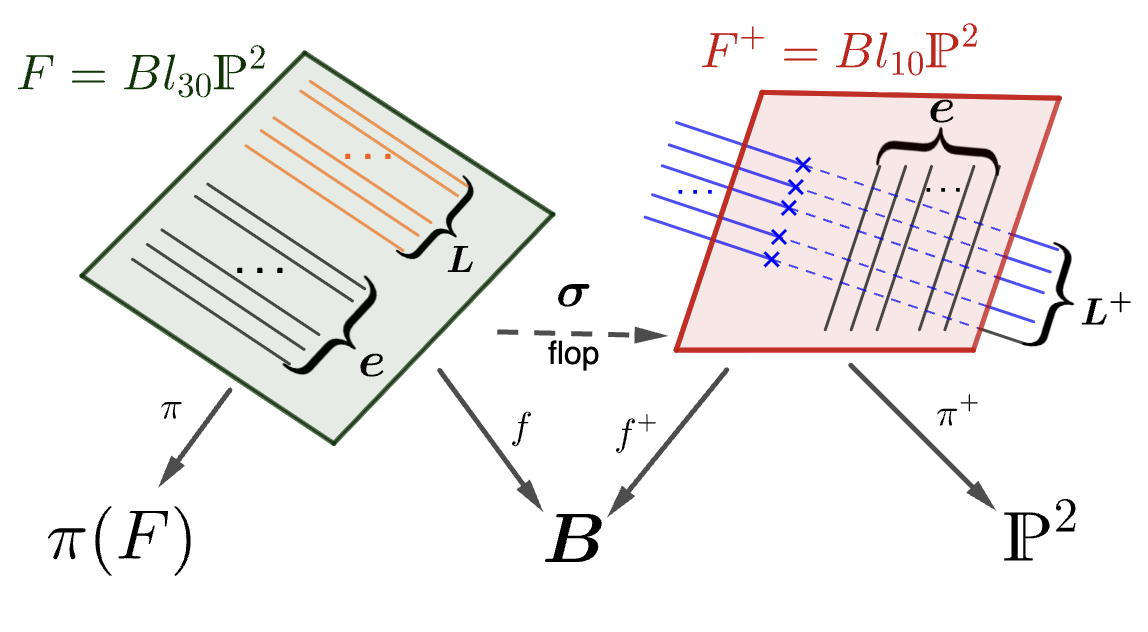}
\caption{Diagram \ref{X1DIAGRAM} restricted to the twenty 4-secant lines to $C$.}
\end{figure}

\section{Mori Program for $Bl_C \PP^3$, with $C$ a semicanonical curve}\label{S5}

\noindent
This section describes the stable base locus decomposition and the birational models of $X=Bl_C\PP^3$ when $C$ is general in the locus of \textit{semicanonical} curves; that is, curves of degree 10 and genus 11 satisfying $\OO_C(2)\cong\omega_C$. Many computations are similar to those in Section \ref{sec::3} and will be omitted.

\medskip\noindent    
Set $X:=Bl_C \PP^3$,  with $C$ a semicanonical curve. By \cite{lamy} $X$ is a weak Fano threefold, which means that $-K_{X}$ is nef and big but not ample. It turns out that it is extremal in the nef cone $\mathrm{Nef}
(X)$. In fact, the arguments of the previous section can also describe the stable base locus decomposition (SBLD) of $X$.

\begin{prop}
    Following the notation \ref{notat}, the stable base locus decomposition of the effective cone
of $X$ consists of the following 4 chambers:
    \begin{enumerate}[(1)]
        \item In the region $(H, E]$, the stable base locus equals the support of $E$.
        \item In the closed cone $[4H-E, H]$, the stable base locus is empty. 
        \item In the region $[11H-3E, 4H-E)$, the stable base locus equals the union of the twenty 4-secant lines of $C$.
        \item In the region $[40H-11E, 11H-3E)$, the stable base locus is equal to the support of $40H-11E$.
    \end{enumerate}
\end{prop}
\subsection{Birational models of $X$.}
We now describe the birational models of $X$ induced by each SLBD chamber in $\MOV(X)$. The only difference with Theorem \ref{teorema1} is written in the following lemma.

\begin{lem}\label{lemma:: doublecover}
    Let $f_0 : X\to Y_0\subset\PP^4$ be the morphism induced by the linear system $|\!-\!K_X|$. Then $f_0$ is generically $2:1$ onto a quadric threefold. 
\end{lem}
\begin{proof}  Since the ideal of the curve $C$ is generated by quartics (Lemma \ref{lem::genInDegX}), then $f_0$ is a morphism. Let us argue that it is generically $2:1$.

Consider a hyperelliptic curve $C'\subset\PP^3$ of bidegree $(2,4)$ in a smooth quadric surface. Given two quartic surfaces $F$ and $F^{\prime}$, without common components and containing  $C^{\prime}$, by \cite{vite}*{Thm. 4.1} we have that $F\cap F^{\prime}=C'\cup C$, where $C$ is a semicanonical curve of degree $10$ and genus $11$. Furthermore, every semicanonical curve is obtained in this fashion. 

Observe that
    $$\OO_{C'}(-K_X)=\OO_{C'}(4H-C)=\OO_{F}(C')|_{C'}=\omega_{C'}.$$
    Therefore, $f_0|_{C'}$ is 2:1 onto a rational curve. Since $X$ is covered by the curves $C'$, this proves that $f_0$ is at least generically 2:1. If $F\in|-K_X|$ is a smooth surface containing such a curve $C'$, observe that $\OO_F(-K_X)=\OO_F(C')$. By \cite[Prop. VIII.3 (iv)]{Beau}, $f_0$ sends $F$ to a degree 2 surface. This, together with $(-K_X)^3=4$, implies $f_0$ is in fact $2:1$, and that the image of $f_0$ is a quadric hypersurface.\\
\end{proof}

\begin{theorem}\label{thm3}
Let $X$ be the blow-up of $\PP^3$ along a generic curve $C\in D_1$. The birational models of $X$ induced by the SBLD of $\EFF(X)$ are the following:
\begin{enumerate}[(1)]
    \item[(0)] $X(H)\cong\PP^3$.
    \item $X(D)\cong X$ for $D\in(H,-K_X)$.
    \item $X(-K_X)= Y$ is a 2:1 ramified cover of a quadric hypersurface $Y_0\subset\PP^4$, and the induced map $f:X\to Y$ has exceptional locus $L$, the union of the twenty 4-secant lines to $C$.
    \item $X(D)= X^{+}\cong X$ for $D\in(11H-3E,-K_X)$. The flopping map $\sigma:X \dashedrightarrow X^{+}$ is a birational Galois involution for $\mu$.
    \item $X(11H-3E)\cong \PP^3$.
\end{enumerate}
\end{theorem}
\noindent
The following diagram summarizes the previous result.
\begin{align*}
\xymatrix @!=1pc{ 
 & X\ar[ddl]_{\pi} \ar[ddr]_{f_0} \ar@{-->}[rr]^{\sigma} \ar[dr]^{f}&  &  X^{+} \ar[dl]_{f^{+}} \ar[ddr]^{\pi^+}&\\
& &Y \ar[d]^{\mu} &  \\
\PP^{3}\ar@{-->}@/_0.6cm/[rrrr]& & \ \ \ \ \ \ Y_0\subset \PP^{4} & &\PP^3}
\end{align*}

\medskip
\begin{proof} The situation is similar to the previous section, except for item (3), which is what we address below. 

By Lemma \ref{lemma:: doublecover}, the linear system $|-K_X|$ induces a $2:1$ map onto a quadric $Y_0$ in $\PP^4$. It follows from \cite[pp. 776]{isk}, that $X$ is \textit{hyperelliptic}, $Y_0$ smooth and one has a commutative diagram
$$ \xymatrix{
X \ar[rr]^{f}\ar[rd]^{f_0}& & Y\ar[ld]_{\mu}\\
& Y_0
}  $$
where $f_0$ denotes the anticanonical morphism, $Y$ is the anticanonical model and $\mu$ a finite morphism of degree $2$. By \cite[Thm. 2.2]{isk} $X$ is uniquely determined by the pair $(Y_0, R_0)$, where $R_0\subset Y_0$ is the branch divisor of $\mu$. In this case, $R_0=Y_0\cap V$, where $V$ is a quartic in $\PP^{4}$. 

The rest of item (3) follows from the previous paragraph and \cite[Lem. 2.8]{jahnke}.\\
\end{proof}

\medskip\noindent
The birational geometry in Theorem \ref{thm3} is similar to that in Theorem \ref{teorema1}, but there are two noteworthy differences which we now describe. First, there is no irreducible surface of degree eight containing the 4-secant lines to a general $C\in \mathcal{H}_{10,11}$.

\begin{cor}\label{Ram0}
There exists an irreducible surface in $\PP^3$ of degree 8 and double along $C$, which  contains the 4-secant lines to $C$.
\end{cor}
\begin{proof}
Using the notation from Theorem \ref{thm3}, the strict transform $R$ of the ramification divisor of $\mu$ has class equal to $8H-2E\in \mathrm{Pic}(X)$. Since $f|_R:R\to R_0$ is birational, and $R_0$ is irreducible for the general $C\in D_1$, so is $R$.

Each of the 4-secant lines $\ell_i$ to $C$ is contracted by both $f$ and $f_0$. Since $X$ is weak Fano, then $Y$ is normal and thus $f^{-1}(f(\ell_i))=\ell_i$. Therefore, $\mu^{-1}(f_0(\ell_i))=f(\ell_i)$ is a single point. This implies that the union $L=\bigcup_{i=1}^{20}\ell_i\subset X$ is contained in $R$. Therefore the image of $R$ under $\pi$ is an irreducible surface of degree 8 and double along $C$, that contains the twenty $4$-secant lines to $C$.\\
\end{proof}

\begin{lem}\label{lem::smoothRam}
    If $C\in D_1$ is general, then $R\subset X$ from the previous proof is smooth.
\end{lem}

\begin{proof}
Let $R'=\pi(R)\subset\PP^3$. We claim that, for a general $C\in D_1$, the following properties are satisfied:
    \begin{enumerate}[a)]
        \item the singular scheme $K_1:=Sing(R')$ is supported on $C$, and
        \item the singular scheme $K_2:=Sing(K_1)$ is reduced and zero-dimensional.
    \end{enumerate}
    Suppose that some $C\in D_1$ satisfies these properties. Let us prove that blowing-up $C$ resolves the singularities of $R'$.

    Let $p\in C$. Consider the completion $\hat{\OO}_{\PP^3,p}$ of the local ring of $\PP^3$ at $p$, which is isomorphic to the ring of power series $\mathbb{C}[[x,y,z]]$. One can find an isomorphism $\hat{\OO}_{\PP^3,p}\cong\mathbb{C}[[x, y, z]]$ such that $C$ is defined infinitesimally by the ideal $(x,y)$. In these coordinates, the defining equation of $R'$ is a power series of the form $$F=\sum_{i+j+k=2}^\infty a_{ijk}x^iy^iz^k.$$
    Since $R'$ is singular along $C$, we must have $a_{ijk}=0$ whenever $i+j\leq1$.

    It is a straightforward computation to verify that $F$ can be written (up to an additional change of coordinates) by a power series of the form
    \begin{equation*}
        F=\left\{\begin{array}{cc}
            x^2-y^2+\textrm{terms of degree}>2 & \textrm{if }p\in C\setminus K_2,\\[2mm]
            y^2+\lambda x^2z+\textrm{terms of degree}>2 & \textrm{if }p\in K_2.
            \end{array}\right.
    \end{equation*}
    Above, the terms of degree 3 or higher belong to the ideal $(x,y)^2$; and $\lambda\neq0$ in the second case. In both cases, we can verify the smoothness of $R$ along the exceptional fiber over $p$, by looking at the affine chart $(x, y, z)=(x, ux, z)$ of the blow-up of $C$.

    Finally, let us use \texttt{Macaulay2} to show that $C\in D_1$ in Corollary \ref{cor::explicitSemicanonical} satisfies properties a) and b). The code immediately above the statement of Corollary \ref{cor::explicitSemicanonical}, exhibits a curve in $D_1$. The following code now constructs the surface $R'$ and verifies properties a) and b), hence the surfaces $R$ is smooth. Since the condition that $R$ is smooth is an open condition, the general $C\in D_1$ satisfies these conditions as well, as claimed.
    \begin{verbatim}
    L = saturate(ideal(C3_0, C3_1, C3_2, C3_3), C);
    R' = ideal(L_0);
    K1 = saturate trim ideal singularLocus R';
    K2 = saturate trim ideal singularLocus K1;
    
    saturate(K1, C) == ideal(1_P3)		-- K1 is supported on C
    K2 == radical K2			-- K2 is reduced
    codim K2 == 3				-- K2 is zero-dimensional
    \end{verbatim}
\end{proof}
\begin{rmk}
    If $C\in D_1$ is not assumed general, then the previous Lemma is false. For example, in the construction above, consider the unique (locally Cohen-Macaulay) curve $C_0$ in a smooth quadric surface representing the divisor $2L_1+4L_2$, where $L_1$ and $L_2$ are lines in different rulings of the quadric. In this setting, $C\in D_1$ can be smooth semicanonical and $R$ has 2 singular points.
\end{rmk}

\medskip\noindent
By Lemma \ref{lemma:: doublecover}, the linear series $|-K_X|$ yields a double quadric $X\to Y_0\subset \PP^4$, whose branch surface arises in Corollary \ref{Ram0}, and is described below.

\begin{prop}\label{prop:: doublequadric}
    The double quadric $f_0:X\to Y_0\subset \PP^4$ is branched along a degree 8 surface with twenty nodes in a special position.
\end{prop}
\begin{proof}
    Observe that $R_0=f_0(R)\subset\PP^4$ is the branch divisor of $f_0$. Its degree is $deg(R_0)=(4H-E)^2\cdot(8H-2E)=8$.

    If $C\in D_1$ is general, $R$ is smooth. Each 4-secant line $\ell_i$ to $C$ is a $(-2)$-curve of $R$ by adjunction, as $$-2=\ell_i^2+K_R\cdot\ell_i=\ell_i^2.$$
    Therefore, $R_0$ is an irreducible surface of degree 8 with twenty singular double points. 
    
    Note that the defect of $Y$ is $\rho(Y)=1$ as $f$ is a small contraction and $rank \ Pic(X)=2$. On the other hand, $$\rho(Y)= |Sing(R_0)|+ h^0(\PP^4,\mathcal{I}(3)) -35,$$ where $\mathcal{I}$ denotes the ideal sheaf of the singular locus of $R_0$, which consists of twenty points. Consequently, $h^0(\PP^4,\mathcal{I}(3))=16$, which means that the singular points of $R_0$ fail to impose independent conditions on cubics.\\
\end{proof}

\medskip\noindent
A double quadric $X\to Y_0\subset \PP^4$ as above is ramified along a surface with twenty nodes that fail to impose independent conditions on cubics by Proposition \ref{prop:: doublequadric}. 
Double quadrics with these twenty nodes in general position exist, and are birationally rigid \cite{sha}. One might then expect by Proposition \ref{prop:: doublequadric} that birational rigidity fails by specializing the nodes. However, the situation is more complicated as the former double  quadric $Y$ satisfies that $H_4(Y,\mathbb{Z})=\ZZ^2$, while the latter (birationally rigid) has $H_2(Y,\ZZ)\cong\ZZ$. Therefore, they are not diffeomorphic, making a deformation between them impossible.

\medskip\noindent
A multiplicity 2 quadric $Y_0^2$ is an element in $\mathcal{F}$. 
However, the anti-canonical model of $X$, a 2:1 ramified cover of $Y_0$, is determined by the pair $(Y_0,R_0)$ which is not in $\mathcal{F}$. Thus, the specialization of a linear determinantal quartic to a double quadric takes place in a birational modification of the family $\mathcal{F}$. 

\medskip\noindent
This resembles a well-known situation for curves. Let $Y_1\subset \PP^2$ be a smooth plane curve of degree 4, and consider the family $$\mathcal{Y}_t= tY_1+Y_0^2,$$ where $Y_0$ is a fixed smooth conic. The stable limit of $\mathcal{Y}_0$ in the moduli space $\mathcal{M}_3$ is a smooth hyperelliptic curve $Y$. This curve arises as a double cover $Y\to Y_0$ ramified at the eight points $R_0=Y_1\cap Y_0$. The pair $(Y_0,R_0)$ is what determines the stable limit, and belongs to a birational modification of the space of quartic curves $\PP^{14}$.

\begin{rmk}\label{AutK3II}
The ideal $\mathcal{I}_C$ of a general curve $C\in \mathcal{H}_{10,11}$ can be generated by smooth quartics whose Picard rank is 2 and whose automorphism group isomorphic to $\ZZ$ (Remark \ref{AutK3}). A smooth curve $C'$ satisfying $\omega_{C'}=\OO_{C'}(2)$ is also cut out by smooth quartics, but in this case, their Picard rank is 3, and the automorphism group of a general such surface is $$\ZZ/2\ZZ*\ZZ/2\ZZ*\ZZ/2\ZZ.$$

Interestingly, \cite{caro} resolves Gizatullin's question in both cases: when is an automorphism of a smooth quartic surface $F\subset \PP^3$ the restriction of a Cremona transformation of $\PP^3$? In each case, the Sarkisov links discussed earlier yield generators of $\mathrm{Aut}(F)$. 

For example, when $C'$ satisfies $\omega_{C'}=\OO_{C'}(2)$, the Galois involution from Theorem \ref{thm3} (3) generates a $\ZZ/2\ZZ$ factor of the free product $Aut(F)=\ZZ/2\ZZ*\ZZ/2\ZZ*\ZZ/2\ZZ$. We thank Carolina Araujo and Daniela Paiva for pointing this out to us.
    \end{rmk}

\medskip
\section{Mori Program for $Bl_C \PP^3$ with $C$ in a cubic surface}\label{S4}

\noindent
This section determines the stable base locus decomposition and describes the birational models of $X=Bl_C\PP^3$, when $C$ is contained in a smooth cubic surface. In contrast to Section \ref{sec::3}, where the existence of Diagram \eqref{diagprop5.8} in Proposition \ref{SLink1} was known, here we must construct the analogous diagram. This construction is the technical core of the section.

\medskip \noindent
Recall that $D_2\subset\mathcal{H}_{10,11}$ denotes the family of curves in $\mathcal{H}_{10,11}$ that lie on a cubic surface. This family forms an irreducible divisor in $\mathcal{H}_{10,11}$. Throughout this section, we fix a general element $C\in D_2$ and set $X=Bl_C \PP^3$. 

\medskip\noindent
Let $S$ be the smooth cubic surface containing $C$. The incidences of its 27 lines with $C$ will be relevant, and are described below.

\begin{lem}\label{secantes} If $S$ is the smooth cubic surface containing $C$, then its 27 lines intersect $C$ as follows:
\begin{itemize}
    \item 2 skew lines $e_1,e_2$ of $S$ are 5-secant lines to $C$. 
    \item 10 lines of $S$ are 4-secant lines to $C$.
    \item 10 lines of $S$ are 3-secant lines to $C$.
    \item 5 lines of $S$ are 2-secant lines to $C$.
\end{itemize}
\end{lem}
\begin{proof}
Given 6 points $p_1,p_2,\ldots,p_6$ on the plane, one can easily compute the values $(d, a_1,a_2,\ldots,a_6)$ for which a degree $d$ plane curve having multiplicity $a_i$ at $p_i$ becomes a degree 10 and genus 11 curve after embedding the blow-up $\PP^2_{p_1,\ldots,p_6}$ as a cubic surface inside $\PP^3$. This computation shows that, up to automorphisms of the cubic surface, the only solution is $(12,5,5,4,4,4,4)$. This means that a curve $C$ of degree 10 and genus 11 in a smooth cubic surface $S$ has class
$$[C]=12\ell-5e_1-5e_2-4\sum_{i=3}^6e_i$$
where $\ell$ is the pullback of $\OO_{\PP^2}(1)$ and the $e_i$'s denote the classes of the exceptional divisors. Denoting the lines on $S$ as $e_i, \ell_{ij}$ and $c_i$, where
\begin{itemize}
    \item $e_i$ is the exceptional divisor over $p_i$,
    \item $\ell_{ij}$ is the strict transform of the line passing through $p_i$ and $p_j$, and
    \item $c_i$ is the strict transform of the conic passing through each $p_j$ except for $p_i$,
\end{itemize}
then the 27 lines of $S$ have the following incidences with $C$:
\begin{itemize}
    \item $e_1$ and $e_2$ are 5-secant to $C$.
    \item $e_3,\ldots,e_6$, as well as $\ell_{ij}$ for $i,j\geq 3$, are 4-secant to $C$.
    \item $\ell_{ij}$ for $i\leq 2$ and $j\geq 3$, as well as $c_i$ for $i\leq 2$, are 3-secant to $C$.
    \item $\ell_{12}$ as well as $c_i$ for $i\geq3$ are 2-secant to $C$.
\end{itemize}
\end{proof}

\medskip\noindent
The following analysis of the normal bundles of these secant lines in $X$ will be used in proving Theorem \ref{teorema2} (3).

\begin{lem}\label{lem::normalBundleLines}
    Let $\ell\subset X$ be the strict transform of a line in $\PP^3$. If $\ell$ is a
    \begin{enumerate}[a)]
       \item 3-secant line to $C$, then $N_{\ell/X}\cong\OO_{\PP^1}(-1)\oplus\OO_{\PP^1}$;
       \item 4-secant line to $C$, then $N_{\ell/X}\cong \OO_{\PP^1}(-1)^2$;
       \item 5-secant line to $C$, and $C$ is contained in a smooth cubic surface, then $N_{\ell/X}\cong\OO_{\PP^1}(-1)\oplus\OO_{\PP^1}(-2).$
    \end{enumerate}
\end{lem}
\begin{proof}
    Let $F\subset X$ be any smooth surface containing $\ell$. Consider the exact sequence
    \begin{equation}\label{equation:normalBundlesSequence}
        0\to N_{\ell/F}\to N_{\ell/X}\to N_{F/X}|_\ell\to0
    \end{equation}
    where
    
    $$    N_{\ell/F}\cong\OO_{\PP^1}(-2-(K_X+F)\cdot\ell),\quad \mbox{and}\quad N_{F/X}|_\ell\cong\OO_{\PP^1}(F\cdot\ell).$$

   The sequence (\ref{equation:normalBundlesSequence}) splits if
    $$(2F+K_X)\cdot\ell<0.$$
    This number is negative in cases $b)$ and $c)$ for $F\in|11H-3E|$; and in case $a)$ for $F\in|3H-E|$; and in each case the splitting of $N_{\ell/X}$ is as claimed. \\
\end{proof}

\medskip\noindent
Now let us consider the case c) of the previous proposition and further compute another normal bundle.
\begin{lem}\label{lem::normalBundleDirectrix}
    Assume that $C$ is contained in a smooth cubic and let $\ell$ be a 5-secant line to $C$. Let $X_1=Bl_\ell(X)$ stand for the blow-up of $\ell$ and let $E_\ell\subset X_1$ be the exceptional divisor. If $y$ denotes the $(-1)$- curve (the directrix) curve inside $E_\ell\cong\mathbb{F}_1$, then $N_{y/X_1}\cong\OO_{\PP^1}(-1)^2.$
\end{lem}
\begin{proof}
    Similarly to (\ref{equation:normalBundlesSequence}), consider the short exact sequence on normal bundles corresponding to $y\subset E_\ell\subset X_1$. The normal bundle $N_{y/E_\ell}$ is isomorphic to $\OO_y(y)\cong\OO_{\PP^1}(-1)$ since $y^2=-1$ inside $E_\ell$. 
    
    Note that $N_{E_\ell/X_1}|_y$ is just $\OO_y(E_\ell)$. Furthermore, $E_\ell$ is a projective bundle over $\ell$. We get that
    $$\OO_{E_\ell}(E_\ell)=\OO_{E_\ell}(-1)=\OO_{E_\ell}(-2f-y)$$
    where $f$ denotes the fiber of the projection $E_\ell\to\ell$. Since $f\cdot y=1$ inside $E_\ell$, we see that this line bundle restricted to $y$ is isomorphic to $\OO_{\PP^1}(-2f\cdot y - y^2)=\OO_{\PP^1}(-1)$. 
    
    All this proves that $N_{y/X_1}$ is an extension of two copies of $\OO_{\PP^1}(-1)$, which must be trivial as claimed.\\
\end{proof}

\medskip\noindent
Let us now start computing the SBLD of $X$. Note that, contrary to the previous cases, the class $-K_X$ is not nef. We follow Notation \ref{notat}.

\begin{prop}\label{effectiveNefD3} The nef cone of $X$ is $\NEF(X)=[5H-E,H].$ The cone of effective divisors of $X$ is $\EFF(X)=[ 3H-E,E]$.
\end{prop}
\begin{proof}
Let us elaborate on the nef cone. By Lemma \ref{lem::genInDegX}, $C$ is cut out, as a scheme, by quintic surfaces. This implies that the linear system  $|5H-E|$ is base-point free and the morphism induced by it contracts (the strict transforms of) the 5-secant lines $e_1$ and $e_2$. 
Thus the class $5H-E$ is nef but not ample. The class $H$ is also nef but not ample as it is the pullback of the hyperplane class in $\PP^3$.

Let us now elaborate on the effective cone. The class $E$ is clearly effective and extremal as it is contracted by the blow-up morphism. On the other hand, since $X=Bl_C \PP^3$ and $C$ is contained in a unique cubic surface $S$, then the class $3H-E$ is effective. Let us prove that it is extremal. Consider a 2-secant line $\ell\subset S$ of $C$. The general hyperplane $H_\ell$ containing $\ell$ satisfies that $H_\ell\cap S=\ell\cup C'$, where $C'$ is a conic intersecting $C$ at 8 points. Thus, $(3H-E)\cdot C'=-2.$
Since $C'\subset S$ is a moving curve in $S$, then the extremality of the $3H-E$ follows from the following Lemma. \\
\end{proof}

\begin{lem}\cite[Lem. 4.1]{CoskunChen}
Suppose that $C'$ is a moving curve in an irreducible effective divisor $S$ satisfying $S \cdot C' < 0$. Then $S$ is extremal in the cone of effective divisors and rigid. 
\end{lem}

\begin{figure}[h]
\centering
\includegraphics[scale=.6]{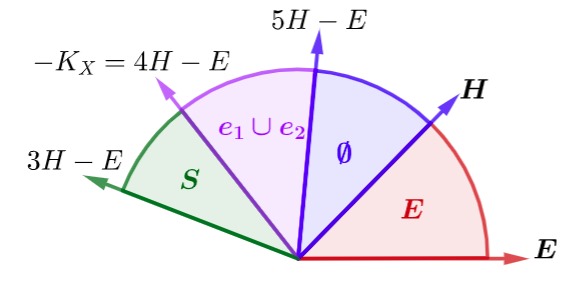}
\caption{ SLBD of $X=Bl_C\ \PP^3$ for $C$ contained in a cubic.}
\end{figure}

\medskip\noindent
Let us describe the stable base locus decomposition of $X$ using the notation in Proposition \ref{SBLDgeneric}.

\begin{prop}\label{SBLDcubic}
    The stable base locus decomposition of the effective cone
of $X$ consists of the following $4$ chambers:
    \begin{enumerate}[(1)]
        \item In the region $(H,E]$, the stable base locus equals the support of $E$.
        \item In the closed cone $[5H-E, H]$, the stable base locus is empty. 
        \item In the region $[4H-E, 5H-E)$, the stable base locus is the union of 5-secant lines $e_1\cup e_2$.
        
        \item In the region $[3H-E, 4H-E)$, the stable base locus equals the cubic surface $S$.    
    \end{enumerate}
\end{prop}
\begin{proof}\ 
    \begin{enumerate}[(1)]
    \item Same argument as Proposition \ref{SBLDgeneric}.
    
    \item This follows from Proposition \ref{effectiveNefD3}.

    \item Since the class $5H-E$ is base-point free, any divisor class $D\in[4H-E, 5H-E)$ satisfies $Bs(D)\subset Bs(4H-E)$. Moreover $D\cdot e_i\leq  0$ for $i=1,2$; thus $e_1\cup e_2\subset Bs(D)$. On the other hand, a smooth quartic $F$ containing $C$ satisfies $S\cap F=C\cup e_1\cup e_2$. Hence, $Bs(4H-E)=e_1\cup e_2$. This proves that $4H-E$ and $5H-E$ span a chamber with stable base locus equal to $e_1\cup e_2$.
    
    \item Following Lemma \ref{secantes}, consider a 2-secant line $\ell\subset S$ to $C$. A general hyperplane $H_\ell$ which contains $\ell$ satisfies that $H_\ell\cap S=\ell\cup C'$, where $C'$ is a conic intersecting $C$ in 8 points. Thus
    $$(4H-E)\cdot C'=0.$$
    Since $S$ is covered by the conics $C'$, this proves that $3H-E$ and $4H-E$ span a chamber with stable base locus equal to $S$.
    \end{enumerate}
\end{proof}

\begin{cor}
    The movable cone of $X$ is  $\MOV(X)=[ -K_X,H]$.
\end{cor}

\medskip\noindent
Note that the class $-K_X$ is base-point free if $C$ is smooth ACM, and no longer nef if $C$ is contained in a cubic surface.

\subsection{The birational models of $X$.} In this section we prove Theorem \ref{teorema2}, which describes the birational models of $X$ induced by each chamber in the SBLD of $\EFF(X)$. We start by describing the first birational modification.

\subsubsection{The inverse flip of $X$.} The nontrivial edge $5H-E$ of the nef cone $\mathrm{Nef}(X)$ induces the following inverse of a flip. We follow closely \cite[Ex. 6.20]{Debarre}. Let $\ell \subset X$ stand for the strict transform of a 5-secant line to $C\subset \PP^3$. The ray $\mathbb{R}^+[\ell]\subset \mathrm{NE}(X)$ fails to be $K_{X}$-negative. However, we will show that $\ell\subset X$ is in fact the exceptional locus of a small contraction $p:X\rightarrow Y$, and that there exists a threefold $X^+$ such that $p$  is the flip of a $(K_{X^+})$-negative extremal contraction $X^+\rightarrow Y$.  In order to prove this, we construct $X^+$ next (see Definition \ref{X1}). 

\medskip\noindent
\textit{Construction:} Consider $X_1:=Bl_{\ell} X$, the blow-up of $X$ along a 5-secant line $\ell$. The exceptional divisor is denoted by $E_\ell$ and the blow-up map by $\beta:X_1\rightarrow X$. By Lemma \ref{lem::normalBundleLines}, $E_{\ell}\cong \mathbb{F}_1$ and the directrix $y\subset E_{\ell}$ has normal bundle $N_{y/X_1}\cong \mathcal{O}_y(-1)^2$. 

\medskip\noindent
Let us now denote by $\alpha:T\rightarrow X_1$ the blow-up of $X_1$ along the directrix $y\subset E_{\ell}$. The exceptional divisor of this blow-up satisfies $E_{y}= \PP^1\times \PP^1$ by Lemma \ref{lem::normalBundleDirectrix}. Observe that the bundle map $\alpha|_{E_y}:E_{y}\rightarrow y$ admits a section $t$ and the class of this curve is $K_T$-negative and extremal. Indeed, $K_T\cdot t = -1$ by adjuntion. On the other hand, the cone of curves contracted by the composition $p\circ\beta \circ \alpha:T\to Y$ is generated by the following three curves: the class of the fiber of $E_{y}\rightarrow y$, the class of the strict transform of the fiber $E_{\ell}\rightarrow \ell$, and $t$. We will call the first two curves $x$ and $z$, respectively. 
This means that the cone generated by these 3 curves is extremal. Hence, if $t$ is not extremal, then it can be written as $$t=ax+bz$$ with $a,b$ positive. This is a contradiction because if we intersect, on both sides of the equation, with $E_y$ we get $-1=a-b$, while if we intersect with $\tilde{E}_{\ell}$ (the strict transform of $E_\ell$) we get $0=-a+b$, which is absurd. 
\begin{definition}
    We denote by $X^+_1$ the image of the contraction map induced by the extremal ray $\mathbb{R}[t]$: $$\mathrm{Cont}_{t}:T\rightarrow X^+_1.$$
\end{definition}

\medskip\noindent
The map $\mathrm{Cont}_t$ is a divisorial contraction and contracts $E_{y}$ into a curve denoted $y^+$. Also, the divisor $E_{\ell}\cong \mathbb{F}_1$ maps under $\mathrm{Cont}_t$ to a surface isomorphic to $\PP^2$.

\medskip\noindent
We compute the normal bundle of $E_\ell^+:=Cont_t(\tilde{E}_\ell)$ in $X_1^+$ next. Before doing that, it is useful to state a general fact about normal bundles of surfaces in threefolds.
\begin{lem}\label{LemmaNormalBundle}
    Let $M$ be a smooth threefold, $F\subset M$ a smooth surface and $C\subset M$ a smooth curve. Denote by $M'=Bl_C(M)$, $F'\subset M'$ the strict transform of $F$, $E\subset M'$ the exceptional divisor and $\pi_0:M'\to M$ the projection. Further denote $\pi:F'\to F$ the restriction.
    \begin{enumerate}[(a)]
        \item If $C$ intersects $F$ transversally in a finite number of points then $N_{F'/M'}=\pi^*N_{F/M}$.
        \item If $C\subset F$ then $N_{F'/M'}=\pi^*N_{F/M}\otimes\OO_{F'}(-E).$
    \end{enumerate}
\end{lem}
\begin{proof}
    For case (a), note that $E|_{F'}$ is equal to the exceptional divisor of the blow-up $\pi:F'\to F$. Hence, $K_{F'}=\pi^*(K_F)+E|_{F'}$, so the normal bundle $N_{F'/M'}$ is the line bundle associated to the divisor class
    \begin{align*}
        F'|_{F'}&=K_{F'}-K_{M'}|_{F'}\\
        &=\pi^*(K_F)+E|_{F'}-K_{M'}|_{F'}\\
        &=\pi^*(K_F)-\pi^*(K_M)\\
        &=\pi^*(F|_F).
    \end{align*}
    The argument is similar for case (b), with the difference that $K_{F'}=\pi^*(K_F)$.\\
\end{proof}

\begin{cor}\label{Corollary5.9}
    The normal bundle of $E^+_\ell:=\mathrm{Cont}_t(E_\ell)$ in $X_1^+$ is isomorphic to $\OO_{\PP^2}(-2).$
\end{cor}
\begin{proof}
    It follows from Proposition \ref{lem::normalBundleLines} that $N_{\ell/X}\cong\OO_{\PP^1}(-2)\oplus\OO_{\PP^1}(-1)$. Therefore, $E_\ell\cong\mathbb{F}_1$ has normal bundle $N_{E_\ell/X_1}=\OO_{E_\ell}(E_\ell)\cong\OO_{E_\ell}(-2)\otimes\OO_{E_\ell}(y)$, where $\mathcal{O}_{E_\ell}(-2)$ denotes the pullback of $\mathcal{O}_{\mathbb{P}^2}(-2)$ under the contraction $E_\ell\cong\mathbb{F}_1\to\mathbb{P}^2$ of the (-1)-curve $y$. We can go from $X_1$ to $X_1^+$ by first blowing-up $y$ and then contracting the other ruling inside the exceptional divisor $E_y$ (which is isomorphic to $\PP^1\times\PP^1$ by Lemma \ref{lem::normalBundleDirectrix}). 
    
    By Lemma \ref{LemmaNormalBundle} (b), $N_{\tilde{E}_\ell/T}\cong\mathcal{O}_{\tilde{E}_\ell}(-2)$. We now want to compare $N_{\tilde{E}_\ell/T}$ with $N_{E_\ell^+/X_1^+}$ via $Cont_t$ using Lemma \ref{LemmaNormalBundle} (a). In order to do this, we first want to prove that the class of the curve
    $$t':=\tilde{E}_\ell\cap E_y$$
    is $t$. This curve maps isomorphically into $y$ via $\alpha$, so clearly $x\cdot y=1$ in $E_y$. It suffices to prove that $(t')^2=0$ in $E_y$. To that end, we  write this number as
    $$(t')^2=\tilde{E}_\ell^2\cdot E_y=\tilde{E}_\ell|_{\tilde{E}_\ell}\cdot y,$$
    where the latter intersection occurs in $\tilde{E}_\ell$. Since $\tilde{E}_\ell|_{\tilde{E}_\ell}=c_1 N_{\tilde{E}_\ell/T}=c_1\mathcal{O}_{\tilde{E}_\ell}(-2)$, then $(t')^2=0$.

    Now we come back to the restriction $Cont_t:\tilde{E}_\ell\to E_\ell^+$, which the previous argument exhibits as the blow-down  $\mathbb{F}_1\to\mathbb{P}^2$ of the (-1)-curve. By Lemma \ref{LemmaNormalBundle} (a) we have
    $$\mathcal{O}_{\tilde{E}_\ell}(-2)=N_{\tilde{E}_\ell/T}=Cont_t^*N_{E_\ell^+/X_1^+}.$$
    Since $Cont_t^*:Pic(E_\ell^+)\to Pic(\tilde{E}_\ell)$ is injective, it follows that $N_{E_\ell^+/X_1^+}\cong\mathcal{O}_{\mathbb{P}^2}(-2)$, as claimed.\\
\end{proof}

\medskip\noindent
Note that the morphism $Cont_{t}$ induces a flop $\phi:X_{1}\to X_{1}^{+}$.

\begin{lem}\label{lem::lambdaIsExtremal}
Let $\lambda$ denote the class of a line in $E^+_\ell=\PP^2$. Then $\lambda \subset \mathrm{NE}(X_1^+)$ is an extremal $(K_{X_1^+})$-negative class. 
\end{lem}
\begin{proof} This follows from \cite[pp. 163]{Debarre}.
Indeed, to verify that $\lambda$ is $K_{X_1^+}$-negative, note that $(K_{X_1^+}+E_{\ell}^+)_{|E_{\ell}^+}=K_{E_{\ell}^+}$, and Corollary \ref{Corollary5.9} yields the restriction ${E_{\ell}^+}_{|E_{\ell}^+}=-2\lambda$. Since $K_{E_{\ell}^+}=-3\lambda$, it follows that $K_{X_1^+}\cdot\lambda=-1$.\\ 
\end{proof}

\begin{definition}\label{X1}
 Let $X^{+}$ be the image of the contraction map $\mathrm{Cont}_{\lambda}:X_1^+\rightarrow X^+$ induced by the extremal ray $\mathbb{R}[\lambda]$.
\end{definition}
\begin{lem}\label{extremalCurve1}
The morphism $\mathrm{Cont}_{\lambda}:X_1^+\rightarrow X^+$ contracts the plane $E^+_{\ell}=\PP^2$ to a point of multiplicity $4$.
\end{lem}
\begin{proof}
This follows from Mori's classification of $K_{X_1^+}$-negative extremal contractions on smooth treefolds \cite[pp. 28]{KM}, with respect to which $\mathrm{Cont}_{\lambda}$ is of type E5 by Corollary \ref{Corollary5.9}.\\
\end{proof}

\medskip\noindent 
We summarize the construction of $X^+$ in Figure \ref{flopX1}.
 \begin{figure}[h]
\centering
\includegraphics[scale=.55]{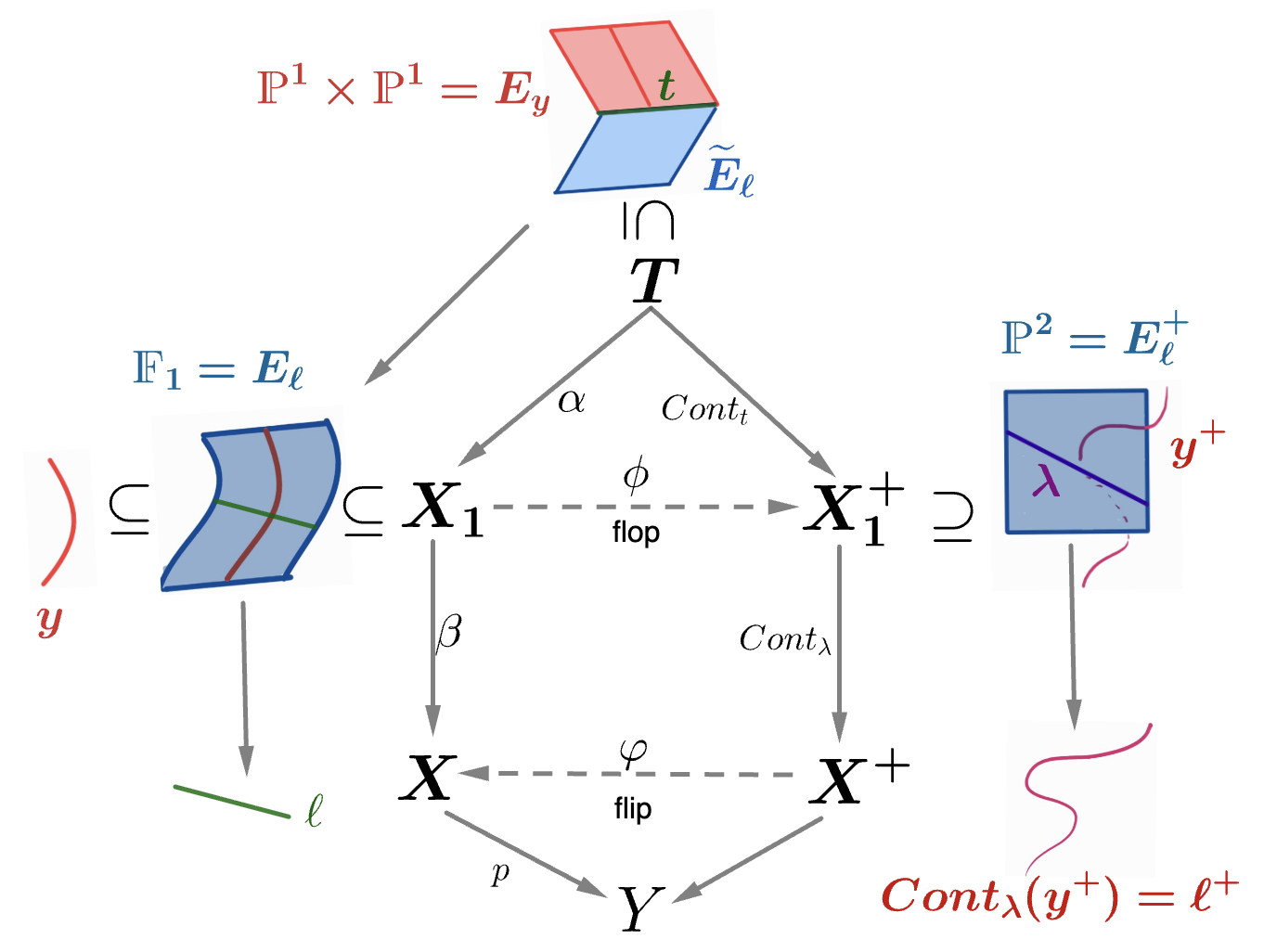}
\caption{The inverse flip of X}\label{flopX1}
\end{figure}
\begin{lem}\label{toY}
    The orthogonal space $\langle\lambda\rangle^\bot\subset N^1(X_1^+)$ is spanned by the classes $$P:=\phi_*(5H-E), \quad \mbox{and } \quad \phi_*(3H-E-E_\ell).$$
Also, $P\cdot y^+=0$, where $y^+$ is the flopped curve of $y\subset E_{\ell}$.

\end{lem}
\begin{proof}
    The classes $5H-E$ and $3H-E-E_\ell$ are linearly independent in $X_1$; therefore, they induce linearly independent classes in $X_1^+$. It then suffices to show that they are orthogonal to $\lambda$.

    In $X$, the class $5H-E$ pairs zero with the 5-secant $\ell$. If $F$ is the strict transform of a general degree 5 surface containing $C$, this means that $F$ and $E_\ell$ are disjoint in $X_1$. To go from $X_1$ to $X_1^+$, one first blows-up the directrix $y\subset E_{\ell}\cong \mathbb{F}_1$ and then contracts the ruling $t$ of the exceptional divisor $E_y\cong\PP^1\times\PP^1$ transversal to the fibers of the projection $E_y\to y$. Since $F$ and $E_\ell$ are disjoint, this process does not modify $F$. Therefore, $\phi_*(F)$ is disjoint to ${E}^+_\ell$ in $X_1^+$. This implies that $P\cdot\lambda =0$. A similar analysis exhibits that $y^+$ is also disjoint from $\phi_*(F)$, which means $P\cdot y^+=0$.

    Similarly, consider the effective representative of $3H-E-E_{\ell}$ which is the strict transform of the cubic surface $S$ containing $C$. We have seen that, inside $X_1$, $S\cap E_{\ell}=y$, where the intersection is transversal. Therefore, the strict transforms of these two surfaces do not intersect after blowing-up $y$. To get to $X_1^+$ we further need to contract the ruling transversal to the fibers of the projection $\PP^1\times\PP^1\cong E_y\to y$. Both $S$ and $E_\ell$ contain a different element in this ruling, thus after contracting them, $S$ and $E_\ell$ remain disjoint. In particular, $\phi_*(3H-E-E_\ell)$ is represented by a surface disjoint to $\lambda\subset E_\ell^+$.
\end{proof}

\medskip\noindent
The threefold $X^+$ is constructed by looking at one $5$-secant $\ell\subset X$. Since there are two such $5$-secant lines in $X$, and they are disjoint (Lemma \ref{secantes}), in what follows $X^+$ is defined by performing the construction accordingly for both of them. This is the inverse flip of $X$ we were aiming at. We prove this claim next and also describe all the birational models of $X$. 

\begin{theorem}\label{teorema2}
Let $X$ be the blow-up of $\PP^3$ along a general curve $C \in D_2$. The birational models of $X$ induced by the SBLD of $\EFF(X)$ are the following.
\begin{enumerate}[(1)]
\item[(0)] $X(H)\cong \PP^3$, and the map $\pi:X\rightarrow \PP^3$ is the blow-up map. 

\item $X(D)\cong X$ for $D\in(5H-E, H)$.

\item $X(5H-E)\cong Y\subset \PP^{15}$ and the map $p:X\rightarrow Y$ is a small contraction whose exceptional locus is $e_1\cup e_2$, the union of the 5-secants to $C$.

\item $X(D)\cong X^{+}$ for $D\in(-K_X,5H-E)$. The rational map $X\dashedrightarrow X^{+}$ is the inverse of a flip with exceptional locus $e_1\cup e_2$, the union of the 5-secants to $C$.

\item $X(4H-E)\cong Y_0$, where $Y_0\subset \PP^4$ is a $\mathbb{Q}$-factorial hypersurface of degree 4 with an elliptic singularity and 10 nodes.
\end{enumerate}
\end{theorem}

The following diagram summarizes the previous result.
\begin{equation*}
\xymatrix @!=1pc{
& X\ar[dl]_{\pi} \ar@{<--}[rr]^{\mbox{\tiny{flip}}} \ar[dr]^{p}&  &  X^{+} \ar[dl]_{p^+}\ar[dr]^{\pi^+}&\\
\PP^{3} && Y&& \ \ \  Y_0\subset \ \PP^{4}} 
\end{equation*}

\begin{proof}\ 
\begin{enumerate}[(1)]
\item Follows from Proposition \ref{effectiveNefD3}, (1).

\item The linear system $|5H-E|$ is base-point free as the ideal of $C$ is generated by quintics. Therefore, it induces a morphism
$$p:X\to Y\subset|5H-E|\cong\PP^{15}.$$
This morphism contracts $e_1\cup e_2$ (Proposition \ref{SBLDcubic}). Conversely, consider an irreducible curve $\gamma\subset X$ contracted by $p$. If $\gamma$ is not contracted by $\pi$, then for any $\varepsilon>0$
$$((5-\varepsilon)H-E)\cdot\gamma=(5H-E)\cdot\gamma - \varepsilon H\cdot\gamma=-\varepsilon\deg\pi(\gamma)<0,$$
so that $\gamma\subset Bs((5-\varepsilon)H-E)=e_1\cup e_2$. Otherwise, $\gamma$ is contracted by $\pi$, thus $H\cdot\gamma=0$ and
$$0=(5H-E)\cdot\gamma=-E\cdot\gamma.$$
This is a contradiction. In conclusion, the exceptional locus of $p$ is $e_1\cup e_2$, as claimed.

\item This item follows closely \cite[Ex. 6.20]{Debarre}. We begin by constructing a morphism $p^+:X^+\to Y$. The linear system $|5H-E|$ is base-point free on $X$ and induces the small contraction $p: X\rightarrow Y$ of the 5-secant lines $e_1,e_2$. Observe that these curves are $K_X$-positive.

Consider the variety $T$ obtained by performing the construction above for each 5-secant line to $C$ (see Figure \ref{flopX1}). Then the linear system of the class $P_0:=(\beta\circ\alpha)^*(5H-E)$ is base-point free and induces the composition $h:T\to X\to Y$. It contracts the curves $t_1$ and $t_2$, where $t_{i}$ is the section of the bundle map $\alpha|_{E_{y_{i}}}:E_{y_{i}}\to y_{i}$, with $y_{i}\subset E_{\ell_{i}}$ the directrix and $\alpha:T\to X_{1}$ the blow-up of $X_{1}$ along the directrices  $y_{1}$ and $y_{2}$. Therefore, it factors through $Cont_t:T\to X_1^+$. Set 
    $$P:=(Cont_t)_*L_0=\phi_*\beta^*(5H-E).$$
It follows from Lemma \ref{toY} that $P\cdot\lambda =0$. Hence, $h$ contracts the fibers of $Cont_\lambda:X_1^+\to X^+$, which implies that $h$ further factors through $X^+$, giving rise to a morphism $p^+:X^+\rightarrow Y$. Also, notice that $D\cdot y^+_i=0 \ (i=1,2)$, which means that $h$ contracts these curves.

\begin{equation*}
\xymatrix @!=1pc{ 
y_1,y_2\subset & X_1\ar[dl] \ar@{-->}[rr]_{\phi}^{\mbox{\tiny{flop}}} &  &  X_1^+\ar[dr]^{g^+} &\ \ \supset y_1^+,y_2^+ \\
X\ar[drr]_{p} \ar@{<--}[rrrr]_{\varphi}^{\mbox{\tiny{flip}}}&&  && X^+\ar[dll]^{p^+} \\
 &&Y &&} 
\end{equation*}
To prove that $\varphi^{-1}:X\dashedrightarrow X^+$ is the inverse of a flip, it suffices to show that if $D$ is any divisor class in the region $(-K_X,5H-E)$, and $D^+:=\varphi^{-1}_* D$, then both $-D^+$ is $p^+$-ample and $D$ is $p$-ample. The class $D$ is $p$-ample because the exceptional locus of $p$ is the union of the two $5$-secant lines $\ell_1 \cup \ell_2$ and $D\cdot \ell_i > 0.$

On the other hand, the exceptional locus of $p^+$ is the union  $y^+_1\cup y_2^+$ by construction. This means that each $\mathbb{R}[y^+_i]$ is an extremal ray. By Lemma \ref{extremalCurve1}, $X^+$ has terminal double points as singularities. Hence $K_{X_1^+}= Cont_\lambda^*K_{X^+}+\tfrac{1}{2}E_{\ell_1}^+ + \tfrac{1}{2}E_{\ell_2}^+$. Therefore
$$K_{X^+}\cdot y^+_i=-\frac{1}{2}.$$
That is, $\mathbb{R}[y_i^+]$ are $K_{X^+}$-negative extremal rays. Since $P\cdot y_i^+=0$, this implies that $D^+\cdot y^+_i<0$, and the result follows.

\item[(4)] First we argue that the linear system $|4H-E|$ induces a birational map
    $$\pi^+:X\dashedrightarrow Y_0\subset\PP^4.$$
Note that $C$ is linked by a complete intersection of two quartics to a reducible curve of the form $C_4\cup e_1\cup e_2$, where $C_4$ denotes a plane curve of degree 4, and $e_1,e_2$ are incident, but not coplanar to it \cite[Thm. 4.1]{vite}. Fix a smooth quartic surface $F$ containing $C$ and $C_4$. Inside $F$ we have:
$$(4H-E)|_F=C+C_4+e_1+e_2.$$
Thus, $|4H-E|$ induces on $C_4$ the canonical linear system, which is injective, plus a base-point at each intersection $C_4\cap e_i$. Note that two general points in $\PP^3$ are always contained in one such plane quartic $C_4$. Therefore, the map induced by $|4H-E|$ is birational onto a hypersurface $Y_0$ of degree $deg(Y_0)=(4H-E)^3=4$.

Let us show that $Y_0\subset\PP^4$ has an elliptic singularity. Observe that
$$(4H-E)\cdot(3H-E)\cdot H=2.$$
If $s$ is a hyperplane section of the cubic surface $S$, then $s\cdot e_1=s\cdot e_2=1$. Consequently, the degree 2 linear system obtained by restricting $|4H-E|$ to $s$ has two base-points. In other words, $4H-E$ contracts $s$. Therefore, it contracts $S$ to a singular point $q\in Y_0$. Now consider a general hyperplane section $M_0\subset Y_0$ through $q$. Its strict transform in $X$ is an element $M\in|H|$ which intersect $S$ along an elliptic curve. The minimal resolution of the singularity $q\in M_0$ is the blow-up $\tilde{M}$ of $M$ at the 10 points of intersection with the 4-secant lines to $C$, and the map $\tilde{M}\to M_0$ is the contraction of the elliptic curve $S\cap M$, which has self-intersection $-1$. This implies that $q\in Y_0$ is an elliptic singularity. Since $Y_0$ is the result of a $K_{X^+}$-extremal contraction, then it is $\mathbb{Q}$-factorial.

Finally, the ten 4-secant lines to $C$ get contracted to ten ordinary double points in $Y_0$.\\
\end{enumerate}
\end{proof}

\end{document}